\documentclass[a4,12pt]{amsart}
\usepackage[utf8]{inputenc}
\usepackage{amsfonts}
\usepackage{latexsym}
\usepackage{amssymb}
\usepackage{amsmath}

\usepackage[dvipsnames]{xcolor}
\usepackage{tikz}
\usepackage{enumerate}
\usepackage{mathrsfs}
\usepackage{amsthm}

\usepackage{hyperref}

\hypersetup{
    colorlinks,
    linkcolor=CadetBlue,
    citecolor=CadetBlue,
    urlcolor=CadetBlue
}

\usepackage{graphicx}

\usepackage{caption}
\usepackage{subcaption}
\numberwithin{equation}{section}

\usepackage[leftcaption]{sidecap}
\sidecaptionvpos{figure}{m}

\usepackage{subcaption}
\usepackage[bottom]{footmisc}

\usepackage{accents}

\usepackage{pgfplots}
\pgfplotsset{compat=1.16}

\tikzset{
    cross/.pic = {
    \draw[rotate = 45] (-#1,0) -- (#1,0);
    \draw[rotate = 45] (0,-#1) -- (0, #1);
    }
}

\usepackage[left=2.2cm, top=1.9cm,bottom=1.9cm,right=2.2cm]{geometry}

\newtheorem{thm}{Theorem}[section]

\newtheorem{lemma}[thm]{Lemma}
\newtheorem{df}[thm]{Definition}

\usepackage{nomencl}
\makenomenclature
\title{Beads on the torus via scaling limits of dimer matchings}
\author{Samuel G. G. Johnston}
\address{Strand Building, King's College London, Strand, London, United Kingdom}

\begin{document}

\maketitle

\begin{abstract}
In a previous article \cite{bead1}, we develop a continuous version of Kasteleyn theory to study the bead model on the torus. These are the point processes on the semi-discrete torus $\mathbb{T}_n := [0,1) \times \{0,1,\ldots,n-1\}$ (thought of as $n$ unit length strings wrapped around a doughnut) with the property that between every two consecutive points on same string, there lies a point on the neighbouring strings. In this companion article, we obtain the main results of the previous article via an alternative route, using scaling limits of dimer models as opposed to the continuous Kasteleyn theory. In any case, we hope that the article may serve as a gentle introduction to Kasteleyn theory on the torus.
\end{abstract}

%\thanks{}

%%%%%%%%%%%%%%%%%%%%%%%%%%%
%%%%%%%%%%%%%%%%%%%%%%%%%%%
%%%%%%%%%%%%%%%%%%%%%%%%%%%
\section{Introduction} \label{sec:introduction} 
%%%%%%%%%%%%%%%%%%%%%%%%%%%
%%%%%%%%%%%%%%%%%%%%%%%%%%%
%%%%%%%%%%%%%%%%%%%%%%%%%%%

\subsection{Introduction and main results}
This paper is a continuation of \cite{bead1}, and as such we will be brisk when redefining notions from that article.

 We begin by recalling from there the definitions surrounding bead configurations on the semi-discrete torus $\mathbb{T}_n$. Let $\mathbb{Z}_n := \mathbb{Z}/n\mathbb{Z}$ be the cyclic group with $n$ elements, and consider the semi-discrete torus $\mathbb{T}_n := [0,1) \times \mathbb{Z}_n$. For $k \geq 1$, a \emph{bead configuration} on $\mathbb{T}_n$ is a collection of $nk$ distinct points on the torus such that there are $k$ points on each string, and the $k$ points on neighbouring strings interlace. More specifically, if $t_1 < \ldots < t_k$ are the points on string $h$, and $t_1'  < \ldots < t_k'$ are the points on string $h+1$ (mod $n$), then we have either
\begin{align} \label{eq:interlace}
t_1 \leq t_1' < \ldots < t_k \leq t_k' \qquad \text{or} \qquad t_1' < t_1 \leq \ldots \leq t_k' < t_k.
\end{align}
See Figure \ref{fig:test0} for a depiction of a bead configuration.

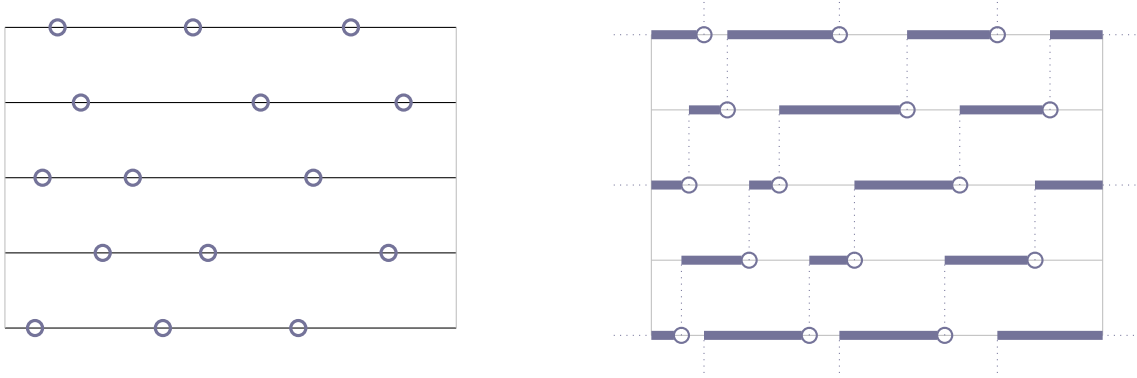
\begin{figure}[h!]
\centering
\begin{subfigure}{.5\textwidth}
\centering
\begin{tikzpicture}
\draw (0,0) -- (6,0);
\draw (0,1) -- (6,1);
\draw (0,2) -- (6,2);
\draw (0,3) -- (6,3);
\draw (0,4) -- (6,4);
\draw[lightgray] (0,0) -- (0,4);
\draw[lightgray] (6,0) -- (6,4);
\draw [very thick, CadetBlue] (0.4,0) circle [radius=0.1];
\draw [very thick, CadetBlue] (2.1,0) circle [radius=0.1];
\draw [very thick, CadetBlue] (3.9,0) circle [radius=0.1];
\draw [very thick, CadetBlue] (5.1,1) circle [radius=0.1];
\draw [very thick, CadetBlue] (1.3,1) circle [radius=0.1];
\draw [very thick, CadetBlue] (2.7,1) circle [radius=0.1];
\draw [very thick, CadetBlue] (0.5,2) circle [radius=0.1];
\draw [very thick, CadetBlue] (1.7,2) circle [radius=0.1];
\draw [very thick, CadetBlue] (4.1,2) circle [radius=0.1];
\draw [very thick, CadetBlue] (1.01,3) circle [radius=0.1];
\draw [very thick, CadetBlue] (3.4,3) circle [radius=0.1];
\draw [very thick, CadetBlue] (5.3,3) circle [radius=0.1];
\draw [very thick, CadetBlue] (0.7,4) circle [radius=0.1];
\draw [very thick, CadetBlue] (2.5,4) circle [radius=0.1];
\draw [very thick, CadetBlue] (4.6,4) circle [radius=0.1];
\draw [white] (2,-0.5) circle [radius=0.1];
\draw [white] (2,5) circle [radius=0.1];
\end{tikzpicture}
  \label{fig:sub10}
\end{subfigure}%
\begin{subfigure}{.5\textwidth}
\centering
\begin{tikzpicture}
\draw[lightgray]  (0,0) -- (6,0);
\draw[lightgray]  (0,1) -- (6,1);
\draw[lightgray]  (0,2) -- (6,2);
\draw[lightgray]  (0,3) -- (6,3);
\draw[lightgray]  (0,4) -- (6,4);
\draw[lightgray] (0,0) -- (0,4);
\draw[lightgray] (6,0) -- (6,4);

\draw [CadetBlue, dotted] (0.4,0) -- (0.4,1); 
\draw [line width=1.2mm, CadetBlue] (0.4,1) -- (1.2,1); 
\draw [CadetBlue, dotted] (1.3,1) -- (1.3,2); 
\draw [line width=1.2mm, CadetBlue] (1.3,2) -- (1.6,2); 
\draw [CadetBlue, dotted] (1.7,2) -- (1.7,3); 
\draw [line width=1.2mm, CadetBlue] (1.7,3) -- (3.3,3);
\draw [CadetBlue, dotted] (3.4,3) -- (3.4,4);
\draw [line width=1.2mm, CadetBlue] (3.4,4) -- (4.5,4);
\draw [CadetBlue, dotted] (4.6,4) -- (4.6,4.5);
\draw [CadetBlue, dotted] (4.6,-0.5) -- (4.6,0);
\draw [line width=1.2mm, CadetBlue] (4.6,0) -- (6,0);
\draw [CadetBlue, dotted] (6.5,0) -- (6,0);
\draw [line width=1.2mm, CadetBlue] (0,0) -- (0.3,0);
\draw [CadetBlue, dotted] (-0.5,0) -- (0,0);

\draw [line width=1.2mm, CadetBlue] (0,2) -- (0.4,2);
\draw [CadetBlue, dotted] (0.5,2) -- (0.5,3);
\draw [line width=1.2mm, CadetBlue] (0.5,3) -- (0.91,3);
\draw [CadetBlue, dotted] (1.01,3) -- (1.01,4);
\draw [line width=1.2mm, CadetBlue] (1.01,4) -- (2.4,4);
\draw [CadetBlue, dotted] (2.5,4) -- (2.5,4.5);
\draw [CadetBlue, dotted] (2.5,-0.5) -- (2.5,0);
\draw [line width=1.2mm, CadetBlue] (2.5,0) -- (3.8,0);
\draw [CadetBlue, dotted] (3.9,0) -- (3.9,1);
\draw [line width=1.2mm, CadetBlue] (3.9,1) -- (5.0,1);
\draw [CadetBlue, dotted] (5.1,1) -- (5.1,2);
\draw [line width=1.2mm, CadetBlue] (5.1,2) -- (6,2);
\draw [CadetBlue, dotted] (6,2) -- (6.5,2);
\draw [CadetBlue, dotted] (-0.5,2) -- (0,2);

\draw [CadetBlue, dotted] (0.7,4) -- (0.7,4.5);
\draw [CadetBlue, dotted] (0.7,-0.5) -- (0.7,0);
\draw [line width=1.2mm, CadetBlue] (0.7,0) -- (2,0);
\draw [CadetBlue, dotted] (2.1,0) -- (2.1,1);
\draw [line width=1.2mm, CadetBlue] (2.1,1) -- (2.6,1);
\draw [CadetBlue, dotted] (2.7,1) -- (2.7,2);
\draw [line width=1.2mm, CadetBlue] (2.7,2) -- (4,2);
\draw [CadetBlue, dotted] (4.1,2) -- (4.1,3);
\draw [line width=1.2mm, CadetBlue] (4.1,3) -- (5.2,3);
\draw [CadetBlue, dotted] (5.3,3) -- (5.3,4);
\draw [line width=1.2mm, CadetBlue] (5.3,4) -- (6,4);
\draw [CadetBlue, dotted] (6,4) -- (6.5,4);
\draw [CadetBlue, dotted] (-0.5,4) -- (0,4);
\draw [line width=1.2mm, CadetBlue] (0,4) -- (0.6,4);

\draw [thick, CadetBlue] (0.4,0) circle [radius=0.1];

\draw [thick, CadetBlue] (2.1,0) circle [radius=0.1];
\draw [thick, CadetBlue] (3.9,0) circle [radius=0.1];

\draw [thick, CadetBlue] (5.1,1) circle [radius=0.1];
\draw [thick, CadetBlue] (1.3,1) circle [radius=0.1];
\draw [thick, CadetBlue] (2.7,1) circle [radius=0.1];

\draw [thick, CadetBlue] (0.5,2) circle [radius=0.1];
\draw [thick, CadetBlue] (1.7,2) circle [radius=0.1];
\draw [thick, CadetBlue] (4.1,2) circle [radius=0.1];

\draw [thick, CadetBlue] (1.01,3) circle [radius=0.1];
\draw [thick, CadetBlue] (3.4,3) circle [radius=0.1];
\draw [thick, CadetBlue] (5.3,3) circle [radius=0.1];

\draw [thick, CadetBlue] (0.7,4) circle [radius=0.1];
\draw [thick, CadetBlue] (2.5,4) circle [radius=0.1];
\draw [thick, CadetBlue] (4.6,4) circle [radius=0.1];
\end{tikzpicture}
  \label{fig:sub20}
\end{subfigure}
\caption{On the left we have a bead configuration on $n = 5$ strings with $k = 3$ beads per string. On the right we have its associated occupation process.}
\label{fig:test0}
\end{figure}

We may associate with any bead configuration on $\mathbb{T}_n$ an \emph{occupation process} $(X_t)_{t \in [0,1]}$ as follows. This is best seen by staring at Figure \ref{fig:test0}. Alternatively, starting from the position above any bead on a string, draw a thick line travelling rightwards, and continue until hitting a bead, at which point we jump up a string. We say a string $h$ is \emph{occupied} at time $t \in [0,1]$ if it carries a thick line, and in this case set $h \in X_t$. (We give a more careful description in the sequel.) There is a one-to-one correspondence between non-empty bead configurations and occupation processes, and we exploit this duality throughout the article. 

\begin{df}
Let $(X_t)_{t \in [0,1]}$ be the occupation process associated with an $(n,k)$ configuration. For some $1 \leq \ell \leq n-1$, at each time $t$ the occupation process $X_t$ takes values in the set of subsets of $\mathbb{Z}_n$ of cardinality $\ell$. We call $\ell$ the \emph{occupation number} of the configuration, and such, say the configuration is an $(n,k,\ell)$ configuration.
\end{df}

Take an $(n,k,\ell)$-configuration, i.e.\ a bead configuration on $n$ unit-length strings with $k$ beads per string. For $1 \leq j \leq k$ and $h \in \mathbb{Z}_n$, write $t_{h,j}$ for the position of the $j^{\text{th}}$ on string $h$. We may therefore associate with each bead configuration a collection $(t_{j,h})_{1 \leq j \leq k, 0 \leq h \leq n-1}$ of elements of $[0,1)$; see Figure \ref{fig:test0}. In particular, we can canonically associate the set of $(n,k)$ configurations with a subset $\mathcal{W}^{(n)}_{k,\ell} := \{(t_{i,j})_{1 \leq i \leq k, 1 \leq j \leq n} \in [0,1)^{nk} : \text{bead configuration} \}$ with a subset of $[0,1)^{nk}$. We write
\begin{align*}
\mathrm{Vol}^{(n)}_{k,\ell} := \text{Volume of the subset $\mathcal{W}^{(n)}_{k,\ell}$ of $[0,1)^{nk}$}.
\end{align*}
We package the volumes $\mathrm{Vol}^{(n)}_{k,\ell}$ of the set of $(n,k,\ell)$ configurations in terms of a generating function. To this end, for $\lambda \in \mathbb{C}$, we define
\begin{align*}
g_N^{\lambda}(x_1,\ldots,x_N) := 
\begin{cases}
e^{- \lambda \ell} \qquad &\text{if for some $k,\ell$, $N = nk$, and $x_1,\ldots,x_{nk}$ is a $(n,k,\ell)$ config.}\\
0 \qquad &\text{otherwise}.
\end{cases}
\end{align*}
We clarify that if $(x_1,\ldots,x_{nk})$, so is $(x_{\sigma(1)},\ldots,x_{\sigma(nk)})$ for any permutation $\sigma$ of the indices. For complex parameters $\lambda$ and $T$ we then define the partition function
\begin{align*}
Z_n(\lambda,T) := \sum_{N \geq 0} \frac{T^N}{N!} \int_{\mathbb{T}_n^N} g_N^\lambda(x_1,\ldots,x_N) \mathrm{d}x_1 \ldots \mathrm{d}x_N = \sum_{k \geq 0} \sum_{ 0 \leq \ell \leq n} T^{nk} e^{-\lambda \ell}\mathrm{Vol}^{(n)}_{k,\ell}.
\end{align*}
Our first main result is the following:

\begin{thm} \label{thm:pf0}
Let $\mathrm{Vol}^{(n)}_{k,\ell}$ denote the volume of the set of bead configurations on $n$ unit length strings with $k$ beads per string and occupation number $\ell$. Then
\begin{align} \label{eq:arrau}
Z_n(\lambda,T) := \sum_{ k \geq 0, 0 \leq \ell \leq n} \frac{T^{nk}}{(nk)!} e^{- \lambda \ell} \mathrm{Vol}^{(n)}_{k,\ell} = \frac{1}{2} \sum_{ \theta \in \{0,1\} } (-1)^{(\theta_1+1)(\theta_2+n+1)} \prod_{ w^n = (-1)^{\theta_2} } (e^{Tw} - (-1)^{\theta_1}e^{ - \lambda} ).
\end{align}
\end{thm}

In \cite{bead1}, for $\theta \in \{0,1\}^2$ we define the functions
\begin{align*}
g_N^{\lambda,\theta}(x_1,\ldots,x_N) := 
\begin{cases}
 \frac{1}{2} (-1)^{(\theta_1+k+1)(\theta_2+n+\ell+1)} e^{- \lambda \ell} \qquad &\text{if $N = nk$, and $x_1,\ldots,x_{nk}$ is a $(n,k,\ell)$ config.}\\
0 \qquad &\text{otherwise},
\end{cases}
\end{align*}
and show that they satisfy
\begin{align} \label{eq:gsum0}
g_N^\lambda = \sum_{ \theta \in \{0,1\}^2} g_N^{\lambda,\theta}.
\end{align} 
We now define a probability measure $\mathbf{P}_n^{\lambda,T}$ on bead configurations by setting
\begin{align*}
\mathbf{P}_n^{\lambda,T} \left( \text{There are $nk$ beads, they have locations in $\mathrm{d}x_1,\ldots,\mathrm{d}x_{nk}$} \right) := \frac{ \frac{T^{nk}}{(nk)!}  g^\lambda_{nk}(x_1,\ldots,x_{nk}) }{ Z_n(\lambda,T) } \mathrm{d}x_1 \ldots \mathrm{d}x_{nk}.
\end{align*}
It turns out to be profitable to decompose $\mathbf{P}_n^{\lambda,T}$ as an affine combination of signed measures. (We recall that a signed measure is a countably additive real-valued function on a sigma algebra.) Indeed, for $\theta \in \{0,1\}^2$ we define a signed measure $\mathbf{P}_n^{\lambda,\theta,T}$ on bead configurations on $\mathbb{T}_n$ by setting
\begin{align*}
\mathbf{P}_n^{\lambda,\theta,T} \left( \text{There are $nk$ beads, they have locations in $\mathrm{d}x_1,\ldots,\mathrm{d}x_{nk}$} \right) := \frac{ \frac{T^{nk}}{(nk)!}  g^{\lambda,\theta}_{nk}(x_1,\ldots,x_{nk}) }{ Z_n(\lambda,T) } \mathrm{d}x_1 \ldots \mathrm{d}x_{nk},
\end{align*}
where $Z_{n,\theta}(\lambda,T)$ is a normalising factor so that $\mathbf{P}^{\lambda,\theta,T}$ has unit total mass. According to \eqref{eq:gsum0}, we can then write 
\begin{align} \label{eq:Psum0}
\mathbf{P}_n^{\lambda,T} = \sum_{ \theta \in \{0,1\}^2 } \frac{Z_{n,\theta}(\lambda,T)}{Z_n(\lambda,T)} \mathbf{P}_n^{\lambda,\theta,T}.
\end{align}
Each $\mathbf{P}_n^{\lambda,\theta,T}$ has a highly tractable determinantal structure.

Given a random bead configuration on $\mathbb{T}_n$ distributed according to $\mathbf{P}_n^{\lambda,T}$, we have a stochastic occupation process $(X_t)_{t \in [0,1)}$ taking values in the set of subsets of $\mathbb{Z}_n$ with a fixed (random) cardinality $\ell$. With a view to studying exclusion processes in the sequel, we find it fruitful to show that the measures $\mathbf{P}^{\lambda,\theta,T}$ may be endowed with a richer determinantal structure, allowing us to express the joint probabilities of not just the probability that beads lie at certain locations in $\mathbb{T}_n$, but also that certain locations are occupied or unoccupied by the process $\mathbb{T}_n$. To this end, we have the following definition:

\begin{df} \label{df:Gamma0}
Let $(x_i)_{ i \in \mathcal{B} \cup \mathcal{O} \cup \mathcal{U} }$ be distinct points of $\mathbb{T}_n$ indexed by disjoints sets $\mathcal{B}$, $\mathcal{O}$ and $\mathcal{U}$. 
Given a bead configuration and its associated occupation process, let $\Gamma(x_i : i \in \mathcal{B} \cup \mathcal{O} \cup \mathcal{U} )$ be the event that
\begin{itemize}
\item Each $\mathrm{d}x_i$ with $i \in \mathcal{B}$ contains a bead.
\item Each $x_i$ with $i \in \mathcal{O}$ is occupied.
\item Each $x_i$ with $i \in \mathcal{U}$ is unoccupied.
\end{itemize}
\end{df}

We have the following result:

\begin{thm} \label{thm:corr}
Define the operators $H^{\lambda,\theta,T},K_n^{\lambda,\theta,T}:\mathbb{T}_n \times \mathbb{T}_n \to \mathbb{C}$ by 
\begin{align} \label{eq:anotherH}
H^{\lambda,\theta,T}(y,y') = \frac{T}{n} \sum_{ z^n=(-1)^{\theta_2}} z^{1+h-h'} \frac{e^{ - (\lambda+\theta_1 \pi i +Tz)[t'-t] }}{1 - e^{- (\lambda+\theta_1 \pi i+Tz) }},
\end{align}
and 
\begin{align} \label{eq:anotherK}
K_n^{\lambda,\theta,T}(y,y') := - \frac{1}{n} \sum_{z^n = (-1)^{\theta_2}} \frac{ z^{h-h'} e^{ - (\lambda+\theta_1 \pi i+Tz)([t'-t]+\mathrm{1}_{t'=t}) } }{ 1 - e^{-  (\lambda+\theta_1 \pi i +Tz) } } .
\end{align}
Then with $\Gamma(x_i:i \in \mathcal{B} \cup \mathcal{O} \cup \mathcal{U})$ as in Definition \ref{df:Gamma0} we have 
\begin{align*}
\mathbf{P}_n^{\lambda,\theta,T}(\Gamma(x_i : i \in \mathcal{B} \cup \mathcal{O} \cup \mathcal{U} ))= \left( \det \limits_{i,j \in \mathcal{O} \sqcup \mathcal{U} \sqcup \mathcal{B} } K^*_{i,j} \right) \prod_{i \in \mathcal{B}} \mathrm{d}t_i
\end{align*}
where, setting $K_{i,j} := K_n^{\lambda,\theta,T}(x_i,x_j) $ and $H_{i,j} := H^{\lambda,\theta,T}(x_i,x_j) $ we have 
\begin{align*}
K^*_{i,j} =
\begin{cases}
K_{i,j} \qquad &\text{if $i \in \mathcal{O}$},\\
\delta_{i,j} - K_{i,j} \qquad &\text{if $i \in \mathcal{U}$},\\
H_{i,j} \qquad &\text{if $i \in \mathcal{B}$}.
\end{cases}
\end{align*}
\end{thm}

\subsection{Dimer matchings and discrete bead configurations}
Direct proofs of Theorems \ref{thm:pf0} and \ref{thm:corr} appear in \cite{bead1} using a continuous analogue of Kasteleyn theory. In the remainder of this article, we obtain proofs of these results via alternative means, developing a scaling limit of a dimer model on the torus. To this end, for positive integers $p$, let $\mathbb{Z}_p$ denote the cylic group of order $p$. Let $m =(m_1,m_2)$ be an ordered pair of positive integers, and consider the discrete torus $\mathbb{T}_{m_1,m_2} := \mathbb{Z}_{m_1} \times \mathbb{Z}_{m_2}$. Whenever $(n_1,n_2)$ is an ordered pair of positive integers, we write $[(n_1,n_2)]_m$ or simply $[(n_1,n_2)]$ for the residue class of $(n_1,n_2)$ modulo $(m_1,m_2)$. Let $\mathrm{Sym}(\mathbb{T}_{m})$ denote the group of bijections of $\mathbb{T}_{m}$. We say $\sigma$ in $\mathrm{Sym}(\mathbb{T}_{m})$ is a \emph{dimer matching} if for every $x \in \mathbb{T}_{m}$ we have 
\begin{align*}
\sigma(x) \in \{ x , [x + \mathbf{e}^1], [x + \mathbf{e}^2] \},
\end{align*}
where $\mathbf{e}^1 := (1,0)$, $\mathbf{e}^2 := (0,1)$. 

These random permutations on $\mathbb{Z}_{m_1} \times \mathbb{Z}_{m_2}$ may be thought of as discrete analogues of bead configurations on the semi-discrete torus $\mathbb{T}_n := [0,1) \times \mathbb{Z}_n$. Indeed, given a permutation $\sigma:\mathbb{Z}_{m_1} \times \mathbb{Z}_{m_2} \to \mathbb{Z}_{m_1} \times \mathbb{Z}_{m_2}$, we can label the squares of $\mathbb{Z}_{m_1} \times \mathbb{Z}_{m_2}$ according to the following rule
\begin{itemize}
\item $x$ is white if $\sigma(x) = x$.
\item $x$ has a solid circle if $\sigma(x) = x+  \mathbf{e}_1$.
\item $x$ has a hollow circle if $\sigma(x) = x + \mathbf{e}_2$.
\end{itemize}
The hollow circles in $\mathbb{Z}_{m_1} \times \mathbb{Z}_{m_2}$ may be thought of as discrete analogues of beads, and the squares with solid circles may be thought of as demarcating an occupation process.

\begin{figure}[h!] 
\centering \begin{tikzpicture}
\draw[step=1cm,gray,very thin] (0,0) grid (12,5);

\draw [CadetBlue, thick] (3.5,0.5) circle [radius=0.3];
\draw [CadetBlue, thick] (5.5,0.5) circle [radius=0.3];
\draw [CadetBlue, thick] (11.5,0.5) circle [radius=0.3];

\draw [CadetBlue, thick] (2.5,1.5) circle [radius=0.3];
\draw [CadetBlue, thick] (4.5,1.5) circle [radius=0.3];
\draw [CadetBlue, thick] (10.5,1.5) circle [radius=0.3];

\draw [CadetBlue, thick] (0.5,2.5) circle [radius=0.3];
\draw [CadetBlue, thick] (3.5,2.5) circle [radius=0.3];
\draw [CadetBlue, thick] (7.5,2.5) circle [radius=0.3];

\draw [CadetBlue, thick] (1.5,3.5) circle [radius=0.3];
\draw [CadetBlue, thick] (3.5,3.5) circle [radius=0.3];
\draw [CadetBlue, thick] (10.5,3.5) circle [radius=0.3];

\draw [CadetBlue, thick] (2.5,4.5) circle [radius=0.3];
\draw [CadetBlue, thick] (5.5,4.5) circle [radius=0.3];
\draw [CadetBlue, thick] (11.5,4.5) circle [radius=0.3];

\draw[CadetBlue, fill, thick]  (2.5,0.5) circle [radius=0.3];
\draw[CadetBlue, fill, thick]  (3.5,1.5) circle [radius=0.3];
\draw[CadetBlue, fill, thick]  (3.5,4.5) circle [radius=0.3];
\draw[CadetBlue, fill, thick]  (4.5,4.5) circle [radius=0.3];
\draw[CadetBlue, fill, thick]  (5.5,1.5) circle [radius=0.3];
\draw[CadetBlue, fill, thick]  (6.5,1.5) circle [radius=0.3];
\draw[CadetBlue, fill, thick]  (7.5,1.5) circle [radius=0.3];
\draw[CadetBlue, fill, thick]  (8.5,1.5) circle [radius=0.3];
\draw[CadetBlue, fill, thick]  (9.5,1.5) circle [radius=0.3];
\draw[CadetBlue, fill, thick]  (10.5,2.5) circle [radius=0.3];
\draw[CadetBlue, fill, thick]  (11.5,2.5) circle [radius=0.3];
\draw[CadetBlue, fill, thick]  (0.5,3.5) circle [radius=0.3];
\draw[CadetBlue, fill, thick]  (1.5,4.5) circle [radius=0.3];
\draw[CadetBlue, fill, thick]  (11.5,2.5) circle [radius=0.3];

\draw[CadetBlue, fill, thick]  (4.5,2.5) circle [radius=0.3];
\draw[CadetBlue, fill, thick]  (5.5,2.5) circle [radius=0.3];
\draw[CadetBlue, fill, thick]  (6.5,2.5) circle [radius=0.3];
\draw[CadetBlue, fill, thick]  (7.5,3.5) circle [radius=0.3];
\draw[CadetBlue, fill, thick]  (8.5,3.5) circle [radius=0.3];
\draw[CadetBlue, fill, thick]  (9.5,3.5) circle [radius=0.3];
\draw[CadetBlue, fill, thick]  (10.5,4.5) circle [radius=0.3];

\draw[CadetBlue, fill, thick]  (11.5,1.5) circle [radius=0.3];
\draw[CadetBlue, fill, thick]  (0.5,1.5) circle [radius=0.3];
\draw[CadetBlue, fill, thick]  (1.5,1.5) circle [radius=0.3];
\draw[CadetBlue, fill, thick]  (2.5,2.5) circle [radius=0.3];
\end{tikzpicture}
\caption{A dimer matching on $\mathbb{T}_m$ with $m_1=12$ and $m_2=5$. The squares containing a hollow (resp. solid) circle correspond to the $x \in \mathbb{T}_m$ for which $\sigma(x) = x + \mathbf{e}_2$ (resp. $\sigma(x) = x + \mathbf{e}_1$). The empty squares are fixed by $\sigma$.}
\label{fig:beethoven}
\end{figure}
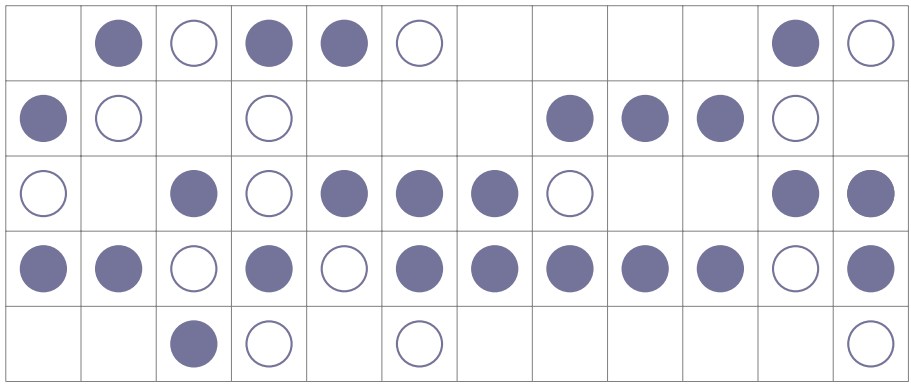

In parallel to our study of bead configurations, we will be interested in counting the number of dimer matchings of $\mathbb{T}_{m_1,m_2}$, and thereafter in studying properties of random dimer matchings. In this direction we introduce the generating function

\begin{align} \label{eq:PF20}
Y_{m_1,m_2}(\alpha,\beta,\gamma) &= \sum_{ \sigma \text{ dimer matching on $\mathbb{T}_{m_1,m_2}$}} \alpha^{ \# \{ x : \sigma(x) = x \} } \beta^{ \# \{ x : \sigma(x) = [x + \mathbf{e}^1] \} } \gamma^{ \# \{ x : \sigma(x) = [x + \mathbf{e}^2] \} } \nonumber \\  
&= \sum_{ \sigma \text{ dimer matching on $\mathbb{T}_{m_1,m_2}$}} \alpha^{ \# \text{white}} \beta^{ \# \text{blue}} \gamma^{ \# \text{black}} .
\end{align}
In Section \ref{sec:kasteleyn} we lay the groundwork, developing a discrete Kasteleyn theory on the torus from scratch. In the following Section \ref{sec:pf0proof} we prove Theorem \ref{thm:pf0}, and thereafter in Section \ref{sec:corrproof} we prove Theorem \ref{thm:corr}.

Again in parallel to the continuous bead case, we will be interested in signed analogues of this partition function. We define the \emph{type} $h(\sigma) := (h(\sigma)_1,h(\sigma)_2)$ in $\mathbb{Z}_{ \geq 0}^2$ of a dimer matching by
\begin{align} \label{eq:exit}
h(\sigma)_i := \# \{ x \in \mathbb{T}_{m} : \sigma(x) = [x + \mathbf{e}^i] \text{ and } \sigma(x)_i = 0 \}.
\end{align}
That is, $h(\sigma)_1$ (resp. $h(\sigma)_2$) counts the number of times a cycle of $\sigma$ exits and reenters the torus $\mathbb{T}_{m_1,m_2}$ horizontally (resp. vertically). We set 
\begin{align} \label{eq:PF20theta}
Y_{m_1,m_2,\theta}(\alpha,\beta,\gamma) &= \sum_{ \sigma \text{ dimer matching on $\mathbb{T}_{m_1,m_2}$}} (-1)^{\theta_1 h_1(\sigma) + \theta_2 h_2(\sigma)} \alpha^{ \# \{ x : \sigma(x) = x \} } \beta^{ \# \{ x : \sigma(x) = [x + \mathbf{e}^1] \} } \gamma^{ \# \{ x : \sigma(x) = [x + \mathbf{e}^2] \} },
\end{align}
and define a signed measure of unit mass on dimer matchings by setting
\begin{align} \label{eq:PF20theta}
P_{m_1,m_2,\theta}^{\alpha,\beta,\gamma}(\sigma) := \frac{ (-1)^{\theta_1 h_1(\sigma) + \theta_2 h_2(\sigma)} \alpha^{ \# \{ x : \sigma(x) = x \} } \beta^{ \# \{ x : \sigma(x) = [x + \mathbf{e}^1] \} } \gamma^{ \# \{ x : \sigma(x) = [x + \mathbf{e}^2] \} } }{ Y_{m_1,m_2,\theta}(\alpha,\beta,\gamma) }.
\end{align}

Now given a dimer matching $\sigma:\mathbb{T}_{m,n} \to \mathbb{T}_{m,n}$, we may associate a bead configuration and its associated occupation process $(X_t)_{t \in [0,1)}$ by setting 
\begin{itemize}
\item $ \sigma(x,h) = (x,h) + \mathbf{e}_2 \implies (x/m, h) \text{ is a bead},$
\item $ \sigma(x,h) = (x,h) + \mathbf{e}_1 \implies (s, h) \text{ is occupied for all $s \in \left[\frac{x}{m},\frac{x+1}{m}\right)$},$
\item $
 \sigma(x,h) = (x,h) \implies (s, h) \text{ is unoccupied for all $s \in \left[\frac{x}{m},\frac{x+1}{m}\right)$}.
$
\end{itemize}
Strictly speaking, this creates a slightly generalised form of a bead configuration, where the inequalities in \eqref{eq:interlace} may not be strict.
Given this association, we will consider a scaling limit of $\mathbb{T}_{m_1,m_2}$ and the weights $\alpha,\beta$ and $\gamma$ in which the horizontal coordinate $m_1$ of the rectangle are sent to infinity while the vertical coordinates remain fixed. Namely, we will be interested in the scaling limit involving constant weights $A = (\alpha_x,\beta_x,\gamma_x : x \in \mathbb{T}_{m_1,m_2})$ with $\alpha_x = \alpha$, $\beta_x = \beta$ and $\gamma_x = \gamma$, where  
\begin{align} \label{eq:scaling}
m_1 = m, m_2 = n \qquad \text{and} \qquad \alpha = 1, \beta = 1 -  \frac{\lambda}{m}, \gamma = \frac{T}{m} \qquad \text{with $m$ even and $m \to \infty$}.
\end{align}
Under this scaling limit, if $Z_n(\lambda,T)$ is the partition function for continuous bead configurations set out in the introduction, then we have
\begin{align} \label{eq:scapf}
  Y_{m,n}\left(1,1-\frac{\lambda}{m},\frac{T}{m} \right) \to Z_n(\lambda,T)  \qquad \text{as $m \to \infty$}
\end{align}
and moreover if $\mathbf{P}_n^{\lambda,\theta,T}$ is the signed probability measure on bead configurations, we have the convergence in distribution
\begin{align*}
P_{m,m,\theta}^{1,1-\frac{\lambda}{m},\frac{T}{m}} \Rightarrow \mathbf{P}_n^{\lambda,\theta,T} \qquad \text{as $m \to \infty$}.
\end{align*}
As such, the remainder of the article is structured as follows.
\begin{itemize}
\item In Section \ref{sec:discrete} we develop a general discrete Kasteleyn theory for the dimer matchings of the torus $\mathbb{T}_{m_1,m_2}$ with spatially dependent weights. Here we show that partition functions and probability measures may be expressed in terms of determinants of certain operators on $\mathbb{T}_{m_1,m_2}$.
\item In the following Section \ref{sec:constant} we specialise to the case where the weights are spatially homogeneous, giving us explicit formulas for the partition function $Y_{m,n}(\alpha,\beta,\gamma)$ and the law of the probability measures $\mathbf{P}_{m_1,m_2}^{\alpha,\beta,\gamma}$. 
\end{itemize}
Thereafter, we have two major tasks over the remainder of the article corresponding to understanding the asymptotics of first the partition function and then the correlation function under the scaling limit set out in \eqref{eq:scaling}. More specifically, 
\begin{itemize}
\item In Section \ref{sec:pf0proof} we study the large $m$ asymptotics of the partition function in Theorem \ref{thm:kast2} under the scaling limit \eqref{eq:scaling}.
\item In Section \ref{sec:corrproof} we study the large $m$ asymptotics of the correlation functions. The main task here is understanding the large-$m$ behaviour of the inverse operator $K_\theta^{-1}(x,y)$ in \eqref{eq:inverse} under the scaling limit \eqref{eq:scaling} where the horizontal distance between $y$ and $x$ is of the order $m$.
\end{itemize}

%%%%%%%%%%%%%%%%%%%%%%%%%%%
%%%%%%%%%%%%%%%%%%%%%%%%%%%
%%%%%%%%%%%%%%%%%%%%%%%%%%%
\section{Kasteleyn theory on the torus} \label{sec:discrete}
%%%%%%%%%%%%%%%%%%%%%%%%%%%
%%%%%%%%%%%%%%%%%%%%%%%%%%%
%%%%%%%%%%%%%%%%%%%%%%%%%%%

In this section we develop a Kasteleyn theory for dimer matchings on the semi-discrete torus. None of this section is new, with everything appearing in some shape or form in the literature. However, we believe our exact formulation of dimer matchings in terms of the semi-discrete torus is more transparent than the more common lozenge tilings (to which it is equivalent). 

%%%%%%%%%%%%%%%%%%%%%%%%%%%
\subsection{The discrete bead model and discrete snake configurations} \label{sec:kasteleyn}

While in the introduction we considered partition functions and (signed) probability measures on dimer matchings associated with constant directional weights, we begin discussion by broadening our scope to spatially dependent weights. Indeed, suppose now we have a collection of $3m_1m_2$ complex parameters $A := ( (\alpha_x,\beta_x,\gamma_x) : x \in \mathbb{T}_{m_1,m_2} )$, and consider the partition function
\begin{align} \label{eq:pf}
Y_{m_1,m_2}(A) := \sum_{ \sigma \in \mathrm{Sym}(\mathbb{T}_{m_1,m_2}) } \prod_{x \in \mathbb{T}_{m_1,m_2}}\left( \alpha_x \mathbf{1}_{\sigma(x) = x} + \beta_x \mathbf{1}_{ \sigma(x) = x + \mathbf{e}^1 } + \gamma_x \mathbf{1}_{ \sigma(x) = x + \mathbf{e}^2 } \right).
\end{align}
We note that the sum in \eqref{eq:pf} is supported only on dimer matchings. We note that up to signs occuring in the sum, the partition function $Y_{m_1,m_2}(A)$ is structurally similar to the determinant of a certain operator. Indeed, define the \emph{Kasteleyn operator} $w_A: \mathbb{T}_{m_1,m_2} \times \mathbb{T}_{m_1,m_2} \to \mathbb{R}$ by 
\begin{align} \label{eq:ko}
w_A(x,y) := \alpha_x \mathbf{1}_{y = x} + \beta_x \mathbf{1}_{ y = x + \mathbf{e}^1 } + \gamma_x \mathbf{1}_{ y = x + \mathbf{e}^2 }
\end{align} 
Then the partition function, which may now be written $Y_{m_1,m_2}(A) := \sum_\sigma \prod_{x \in \mathbb{T}_{m_1,m_2}} w_A(x,y)$, bears striking similarity with the determinant 
\begin{align*}
\det\limits_{x,y \in \mathbb{T}_{m_1,m_2}} w_A(x,y) = \sum_{ \sigma } \mathrm{sgn}(\sigma)  \prod_{x \in \mathbb{T}_{m_1,m_2}} w_A(x, \sigma(x)),
\end{align*}
of the operator $w_A$. Kasteleyn theory is, loosely speaking, the art of converting partition functions into the sums of determinants of operators. The trick is to consider a collection of four operators $\{K_{A,\theta} : \theta \in \{0,1\}^2 \}$ that are modifications of $w_A$ involving additional negative signs for exiting and reentering the torus in the horizontal and/or vertical direction. These modifications are given by 
\begin{align} \label{eq:Kdef}
K_{A,\theta} (x,y) := (-1)^{ \theta_1 \mathbf{1}\{x_1 = m_1-1, y_1 = 0\} + \theta_2 \mathbf{1}\{ x_2 = m_2 - 1, y_2 = 0\} } w_A(x,y), \qquad x,y \in \mathbb{T}_{m_1,m_2}.
\end{align}
Of course then $K_{A,(0,0)} = w_A$. We now have the following crucial result, which states that the partition function $ Y_{m_1,m_2}(A)$ may be written as a sum of determinants of the operators $\{K_{A,\theta} : \theta \in \{0,1\}^2 \}$:
%%%%%%%%%%%%%%%%%%%%%%%%%%%%%%%
\begin{thm} \label{thm:kast}
%%%%%%%%%%%%%%%%%%%%%%%%%%%%%%%
The partition function $Y_{m_1,m_2}(A)$ associated with dimer matchings on the torus $\mathbb{T}_{m_1,m_2}$ may be written in terms of a sum of four determinants. More specifically, we have $Y_{m_1,m_2}(A) = \sum_{\theta \in \{0,1\}^2 } Z_{m,\theta}(A)$, where 
\begin{align} \label{eq:zz}
Y_{m_1,m_2,\theta}(A) = \frac{1}{2}(-1)^{(\theta_1+m_1+1)(\theta_2+m_2+1) } \det\limits_{x , y \in \mathbb{T}_{m_1,m_2}} K_{A,\theta}(x,y)
\end{align}
\end{thm}

Before proving Theorem \ref{thm:kast}, we require a quick lemma, giving an expression for the sign $\mathrm{sgn}(\sigma)$ of a dimer matching $\sigma$ in terms of its type.

\begin{lemma} \label{lem:lazy}
Let $\sigma$ be a dimer matching on $\mathbb{T}_{m_1,m_2}$, and recall that for $i=1,2$, $h(\sigma)_i$ defined in \eqref{eq:exit} counts the number of times the dimer matching exits and reenters the discrete torus on the $i^{\text{th}}$ side. Then
\begin{align} \label{eq:xx}
\mathrm{sgn}(\sigma) = \frac{1}{2} \sum_{\theta \in \{0,1\}^2} (-1)^{ (m_1+\theta_1+1)(m_2+\theta_2+1) } (-1)^{ \theta_1 h(\sigma)_1 + \theta_2 h(\sigma)_2 }.
\end{align}
\end{lemma}

\begin{proof}
Throughout the proof we abbreviate $h_i := h(\sigma)_i$. Consider the decomposition of a dimer matching $\sigma$ into its cycles $C_1,\ldots,C_j$. When $j \geq 1$ (i.e. $\sigma$ contains at least one non-trivial cycle), it is easily seen that each cycle of $\sigma$ must exit the torus vertically and horizontally the same number of times as each of the other cycles. Namely, there exists a pair of non-negative integers $(q_1,q_2) = q(\sigma)$ such that every cycle of $\sigma$ exits horizontally $q_1$ times and vertically $q_2$ times. Moreover, for topological reasons
%%%%%%%%%%%%%% DUBIOUS
it must be the case that either $q_1=1$, $q_2=1$, or both. For every such cycle, the number of elements in the cycle is given by $q_1 m_1 + q_2 m_2$. In particular, the sign of $\sigma$ is given by $(-1)^{ j(q_1m_1+q_2m_2-1)}$. On the other hand, $h_i = jq_i$. In particular, in order to establish \eqref{eq:xx} it remains to prove that the equation 
\begin{align} \label{eq:yy}
(-1)^{j(q_1m_1+q_2m_2-1)} = \frac{1}{2} \sum_{\theta \in \{0,1\}^2} (-1)^{ (m_1+\theta_1+1)(m_2+\theta_2+1) + j ( \theta_1 q_1 + \theta_2 q_2 ) }
\end{align}
holds in either of the cases
\begin{itemize}
\item Case 1: $j=0$ (and it does not matter what values $q_1,q_2$ take).
\item Case 2: $j \geq 1$, and $q_1,q_2$ are a pair of non-negative integers at least one of which is equal to one.
\end{itemize}
To this end, we begin by factorising the dependence on $\theta$ in the exponent in the right-hand-side of \eqref{eq:yy}. Indeed, we have
\begin{align} \label{eq:theta brahms}
&(m_1+\theta_1+1)(m_2+\theta_2+1) + j ( \theta_1 q_1 + \theta_2 q_2 )\nonumber \\
& = (m_1+\theta_1+1 + j q_2 )(m_2+\theta_2+1 + jq_1 ) - jq_1(m_1+1) - jq_2(m_2+1) - j^2q_1q_2. 
\end{align}
Consider the term in \eqref{eq:theta brahms} that depends on $\theta$. The product $ (m_1+\theta_1+1 + j q_2 )(m_2+\theta_2+1 + jq_1 ) $ is odd for one choice of $\theta = (\theta_1,\theta_2) \in \{0,1\}^2$ and even for the other three choices. It follows that 
\begin{align*}
 \frac{1}{2} \sum_{\theta \in \{0,1\}^2} (-1)^{  (m_1+\theta_1+1 + j q_2 )(m_2+\theta_2+1 + jq_1 ) } = \frac{1}{2}(-1+1+1+1) = 1,
\end{align*}
and hence by virtue of \eqref{eq:theta brahms}, 
\begin{align} \label{eq:makela}
\frac{1}{2} \sum_{\theta \in \{0,1\}^2} (-1)^{ (m_1+\theta_1+1)(m_2+\theta_2+1) + j ( \theta_1 q_1 + \theta_2 q_2 ) } = (-1)^{ - j q_1(m_1+1) -  jq_2(m_2+1) + j^2q_1q_2}.
\end{align}
Plugging \eqref{eq:makela} into \eqref{eq:yy} and multiply both sides by $(-1)^{j(q_1m_1+q_2m_2)}$, we see that \eqref{eq:yy} reduces to 
\begin{align} \label{eq:machine}
(-1)^j = (-1)^{j (q_1 + q_2 + jq_1q_2)}.
\end{align}
Clearly \eqref{eq:machine} is true whenever $j$ is even, which in particular proves \eqref{eq:yy} in Case 1. It remains to consider the subcase of Case 2 in which $j$ is odd. Here, supposing without loss of generality that $q_1 = 1$, \eqref{eq:machine} reads $(-1)^j = (-1)^{j(1 + (1+j)q_2}$. Since $j$ is odd, $(1+j)q_2$ is even, and we see that \eqref{eq:machine} is also true in Case 2.
\end{proof}

\begin{proof}[Proof of Theorem \ref{thm:kast}]
Whenever $\sigma$ is a dimer, we note from the definition \eqref{eq:Kdef} of $K_{A,\theta}$, we have
\begin{align} \label{eq:tonnes}
 \prod_{ x \in \mathbb{T}_{m_1,m_2} } K_{A,\theta}(x,\sigma(x)) = (-1)^{\theta_1 h(\sigma)_1 + \theta_2 h(\sigma)_2 } \prod_{ x \in \mathbb{T}_{m_1,m_2} } w_A(x,\sigma(x)).
\end{align}
Consider the quantity $\frac{1}{2} \sum_{\theta \in \{0,1\}^2 } Z_{m,\theta}(A)$. Expanding the determinants to obtain the second equality below, and then using \eqref{eq:tonnes} and interchanging the order of summation to obtain the third, we have 
\begin{align} \label{eq:uu}
\frac{1}{2} \sum_{\theta \in \{0,1\}^2 } Y_{m_1,m_2,\theta}(A) &:= \frac{1}{2}(-1)^{(\theta_1+m_1+1)(\theta_2+m_2+1) } \det\limits_{x , y \in \mathbb{T}_{m_1,m_2}} K_{A,\theta}(x,y)\\
 &=  \sum_{ \sigma \in \mathrm{Sym}(\mathbb{T}_{m_1,m_2}) }  \left\{ \frac{1}{2} \sum_{\theta \in \{0,1\}^2}(-1)^{(\theta_1+m_1+1)(\theta_2+m_2+1) }  \mathrm{sgn}(\sigma) (-1)^{\theta_1 h(\sigma)_1 + \theta_2 h(\sigma)_2 }  \right\} \prod_{ x \in \mathbb{T}_{m_1,m_2} } w_A(x,\sigma(x)). 
\end{align}
As above, we note that $\prod_{ x \in \mathbb{T}_{m_1,m_2} } w_A(x,\sigma(x))$ is non-zero only when $\sigma$ is a dimer matching. In this case, by Lemma \ref{lem:lazy} the term inside the braces is equal to one. Theorem \ref{thm:kast} follows.

\end{proof}

We now turn to considering probabilist aspects of dimer matchings. In the setting when the weights $A = ((\alpha_x,\beta_x,\gamma_x):x \in \mathbb{T}_{m_1,m_2})$ are positive, we may define a probability measure $P_A$ on dimer matchings by setting 
\begin{align*}
P_A(\sigma) := \frac{1}{Y_{m_1,m_2}(A)} \prod_{x \in \mathbb{T}_{m_1,m_2}} w_A(x,\sigma(x)),
\end{align*}
where $w_A$ is the Kasteleyn operator defined in \eqref{eq:ko}. We are now interested in computing probabilities under $P_A$ of the joint events $\{ \sigma(x^1)=y^1 ,\ldots, \sigma(x^p) = y^p\}$, where $x^1,\ldots,x^p,y^1,\ldots,y^p$ are elements of $\mathbb{T}_{m_1,m_2}$ such that for each $i$, $y^i \in \{x^i,[x^i+\mathbf{e}^1],[x^i+\mathbf{e}^2]\}$. Since the partition function $Y_{m_1,m_2}(A)$ can be written as a sum of determinants of operators $K_{A,\theta}$, and each of these determinants is itself a polynomial in the weights, we now develop a quick lemma providing insight for how determinants of matrices depend on the matrix elements.

%We use the following notation. Let $(t_i : i \in I)$ denote a collection of indeterminates indexed by a finite set $I$ and consider a polynomial of the form
%\begin{align*}
%p(t) := \sum_{ S \subseteq I} p_S \prod_{ i \in S} t_i .
%\end{align*}
%Then for distinct $i_1,\ldots,i_j$ in $I$ we write
%\begin{align*}
%\mathrm{Coef}( t_{i_1} \ldots t_{i_j} , p(t) ) := \sum_{ \{ i_1,\ldots,i_j \}  \subseteq S \subseteq I } p_S \prod_{ i \in S - \{i_1,\ldots,i_j\} } t_i.
%\end{align*}

\begin{lemma} \label{lem:detlem}
Let $(x^1,y^1),\ldots,(x^p,y^p)$ be distinct elements of $\{1,\ldots,N\}^2$. Then
\begin{align*}
\frac{ \partial}{\partial M_{x^1,y^1} } \ldots \frac{ \partial}{\partial M_{x^p,y^p} } 
 \det \limits_{i,j = 1}^N M_{i,j} = \det \limits_{i,j = 1}^N M_{i,j}  \det_{i \in Y, j \in X} M^{-1}_{i,j}, 
\end{align*}
where $X = \{ x_1,\ldots,x_p\}$ and $Y = \{ y_1,\ldots,y_p\}$. 
\end{lemma}
\begin{proof}
On the one hand, the coefficient of $M_{x^1,y^1} \ldots M_{x^p,y^p}$ in $\det M$ may be directly given as
\begin{align} \label{eq:jac1}
(-1)^{\sum_{k=1}^p (x^k + y^k) } \det \limits_{i \notin X, j \notin Y} M_{i,j}.
\end{align}
On the other hand, by Jacobi's identity (see e.g. Horn and Johnson \cite[Section 0.8.4]{horn2012matrix}) we have
\begin{align} \label{eq:jac2}
\det \limits_{i \notin X, j \notin Y} M_{i,j} =(-1)^{\sum_{k=1}^p (x^k + y^k) }  \det (M) \det_{i \in Y, j \in X} M^{-1}_{i,j}.
\end{align}
Combining \eqref{eq:jac1} and \eqref{eq:jac2} yields the result.
\end{proof}

\subsection{Basic correlations} \label{sec:basic}
We now state and prove our correlation formula:

\begin{lemma}\label{lem:cf}
Let $x^1,\ldots,x^j,y^1,\ldots,y^j$ be elements of $\mathbb{T}_{m_1,m_2}$ such that for each $i$, $y^i \in \{x^i,[x^i+\mathbf{e}^1],[x^i+\mathbf{e}^2]\}$. For $k \in \{1,2\}$, let $g_k := \{ 1 \leq i \leq j : x^i_k = m_k-1,y^i_k = 0 \}$ be the number of edges $(x^i,y^i)$ that exit and reenter the torus in the $i$-direction. Then the probability that a dimer matching contains edges $(x^1,y^1),\ldots,(x^k,y^k)$ is given by 
\begin{align*}
&P_A \left( \sigma(x^1) = y^1,\ldots,\sigma(x^p) = y^p \right) = \prod_{ i = 1}^k w_A(x^i,y^i) \sum_{ \theta  } \mu_{A,\theta} (-1)^{ \theta_1 g_1 + \theta_2 g_2}  \det \limits_{ i,j = 1}^k \left(  K^{-1}_{A,\theta} ( y^i,x^j) \right),
\end{align*}
where $\mu_{A,\theta} := Z_{m,\theta}(A)/Y_{m_1,m_2}(A)$. 
\end{lemma}

\begin{proof}
For $(x,y) \in \mathbb{T}_{m_1,m_2}$, write $W_{x,y} := w_A(x,y)$. Then for invertible matrices $(M_{i,j})_{1 \leq i,j \leq N}$ we have 
\begin{align} \label{eq:beet}
P_A \left( \sigma(x^1) = y^1,\ldots,\sigma(x^p) = y^p \right)  &:= \frac{1}{Y_{m_1,m_2}(A)} \sum_{ \sigma} \prod_{k=1}^p \mathbf{1} \{ \sigma(x^i) = y^i \} \prod_{ x \in \mathbb{T}_{m_1,m_2}} w_A(x^i,y^i)\nonumber \\
&=  \frac{\prod_{i=1}^k w_A(x^i,y^i)  }{Y_{m_1,m_2}(A)} \frac{\partial}{\partial W_{x^1,y^1} } \ldots \frac{\partial}{\partial W_{x^p,y^p} }Y_{m_1,m_2}(A).
\end{align}
Now using the expression in Theorem \ref{thm:kast} for $Y_{m_1,m_2}(A)$ we have
\begin{align} \label{eq:beet2}
&P_A \left( \sigma(x_1) = y_1,\ldots,\sigma(x_p) = y_p \right)\\ \nonumber
&= \frac{1}{2} \frac{\prod_{i=1}^p w_A(x^i,y^i)   }{Y_{m_1,m_2}(A)} \sum_{ \theta \in \{0,1\}^2 } (-1)^{(\theta_1+m_1+1)(\theta_2+m_2+1) } \frac{\partial}{\partial W_{x^1,y^1} } \ldots \frac{\partial}{\partial W_{x^p,y^p} } \det \limits_{x,y \in \mathbb{T}_{m_1,m_2}} K_{A,\theta}(x,y) .
\end{align}
Write $W^\theta_{x,y} := K_{A,\theta}(x,y)$. Then $W^\theta_{x,y} = (-1)^{\theta_1 \mathbf{1}_{y_1 = 0,x_1=m_1-1} + \theta_2 \mathbf{1}_{y_2 = 0,x_2 = m_2-1} } W_{x,y}$. In particular, with $g_1,g_2$ as in the statement of the lemma, using the relation between $W^\theta_{x,y}$ and $W_{x,y}$ to obtain the first equality below, and Lemma \ref{lem:detlem} to obtain the second, we have 
\begin{align} \label{eq:choc}
\frac{\partial}{\partial W_{x^1,y^1} } \ldots \frac{\partial}{\partial W_{x^p,y^p} } \det \limits_{x,y \in \mathbb{T}_{m_1,m_2}} K_{A,\theta}(x,y) &= (-1)^{\theta_1 g_1 + \theta_2 g_2} \frac{\partial}{\partial W^\theta_{x^1,y^1} } \ldots \frac{\partial}{\partial W^\theta_{x^p,y^p} } \det \limits_{x,y \in \mathbb{T}_{m_1,m_2}} W^\theta_{x,y} \nonumber \\
&= (-1)^{\theta_1 g_1 + \theta_2 g_2}\det \limits_{x,y \in \mathbb{T}_{m_1,m_2}} W^\theta_{x,y} ~ \det \limits_{i,j=1}^p (W^\theta)^{-1}_{y^i,x^j} \nonumber \\
&:= (-1)^{\theta_1 g_1 + \theta_2 g_2}\det \limits_{x,y \in \mathbb{T}_{m_1,m_2}} K_{A,\theta}(x,y)  ~ \det \limits_{i,j=1}^p K_{A,\theta}^{-1}(y^i,x^j) .
\end{align}
Plugging \eqref{eq:choc} into \eqref{eq:beet2}, we obtain the result.

\end{proof}

We remark that alternatively, we can construct a signed measure $P_{A,\theta}$ on the set of dimer matchings on $\mathbb{T}_{m_1,m_2}$, whose correlation functions are given by 
\begin{align} \label{eq:beet2}
P_{A,\theta} \left( \sigma(x^1) = y^1,\ldots,\sigma(x^p) = y^p \right) = (-1)^{ \theta_1 g_1 + \theta_2 g_2}  \prod_{ i = 1}^k w_{A}(x^i,y^i) \det \limits_{ i,j = 1}^k \left(  K^{-1}_{A,\theta} ( y^i,x^j) \right).
\end{align}
With this definition at hand, the measure $P_A$ may be written
\begin{align*}
P_A = \sum_{\theta \in \{0,1\}^2 } \mu_{A,\theta} P_{A,\theta}.
\end{align*}

%%%%%%%%%%%%%%%%%%%%%%%%%%%%%
\section{Translation invariant weights and diagonalisation} \label{sec:constant}
%%%%%%%%%%%%%%%%%%%%%%%%%%%%%
For the remainder of the paper, we now assume that we are in the special case where the weights $A = (\alpha_x,\beta_x,\gamma_x : x \in \mathbb{T}_{m_1,m_2} )$ take the form $\alpha_x = \alpha$, $\beta_x = \beta$ and $\gamma_x = \gamma$ for fixed positive reals $\alpha,\beta,\gamma$. In this special case, the partition function $Y_{m_1,m_2}(A)$ in \eqref{eq:pf} may alternatively be written
\begin{align} \label{eq:PF2}
Y_{m_1,m_2}(\alpha,\beta,\gamma) = \sum_{ \sigma \text{ dimer matching on $\mathbb{T}_{m_1,m_2}$}} \alpha^{ \# \{ x : \sigma(x) = x \} } \beta^{ \# \{ x : \sigma(x) = [x + \mathbf{e}^1] \} } \gamma^{ \# \{ x : \sigma(x) = [x + \mathbf{e}^2] \} }.
\end{align}
Here the operators $K^{\alpha,\beta,\gamma}_{m_1,m_2,\theta}:\mathbb{T}_{m_1,m_2} \times \mathbb{T}_{m_1,m_2} \to \mathbb{R}$ are given by 
\begin{align} \label{eq:Kdefconstant}
K^{\alpha,\beta,\gamma}_{m_1,m_2,\theta} (x,y) := (-1)^{ \theta_1 \mathbf{1}\{x_1 = m_1-1, y_1 = 0\} + \theta_2 \mathbf{1}\{ x_2 = m_2 - 1, y_2 = 0\} }  \left( \alpha \mathbf{1}_{y = x} + \beta \mathbf{1}_{ y = x + \mathbf{e}^1 } + \gamma \mathbf{1}_{ y = x + \mathbf{e}^2 } \right).
\end{align}
In any case, in this regime of translation invariant weights, we are able to diagonalise the operators $K^{\alpha,\beta,\gamma}_{m_1,m_2,\theta}$. Define the inner product 
\begin{align}
\langle f, g \rangle := \sum_{ x \in \mathbb{T}_{m_1,m_2} } f(x) \overline{ g(x) }
\end{align}
on the vector space of complex-valued functions $\{f : \mathbb{T}_{m_1,m_2} \to \mathbb{C} \}$. The following lemma is easily verified. 

\begin{lemma}  \label{lem:diag}
The operator $K^{\alpha,\beta,\gamma}_{m_1,m_2,\theta}$ has an orthonormal basis of eigenfunctions $\{ f_j : j \in \mathbb{T}_{m_1,m_2} \}$ given by 
\begin{align*}
f_j(x) = \frac{1}{ \sqrt{m_1 m_2} } \exp \left\{ 2 \pi i \left( \frac{ j_1 + \theta_1 /2}{ m_1 } x_1 + \frac{ j_2 + \theta_2/2}{ m_2} x_2 \right) \right\}. 
\end{align*}
The eigenvalue of $f_j$ is given by 
\begin{align*}
\alpha + \beta \exp \left\{ 2 \pi i  \frac{ j_1 + \theta_1 /2}{ m_1 } \right\} + \gamma \exp \left\{ 2 \pi i  \frac{ j_2 + \theta_2 /2}{ m_2 } \right\} .
\end{align*}
\end{lemma}

The diagonalisation in Lemma \ref{lem:diag} now furnishes tractable representations of both the determinant and inverse of the operators $K^{\alpha,\beta,\gamma,\theta}_{m_1,m_2}$. Indeed, since the determinant of an operator is the product of its eigenvalues, we have 
\begin{align} \label{eq:det}
\det\limits_{ x,y \in \mathbb{T}_{m_1,m_2} }K^{\alpha,\beta,\gamma}_{m_1,m_2,\theta} (x,y) = \prod_{ j \in \mathbb{T}_{m_1,m_2}} \left( \alpha + \beta \exp \left\{ 2 \pi i  \frac{ j_1 + \theta_1 /2}{ m_1 } \right\} + \gamma \exp \left\{ 2 \pi i  \frac{ j_2 + \theta_2 /2}{ m_2 } \right\} \right).
\end{align} 
In particular, by combining Theorem \ref{thm:kast} and \eqref{eq:det} we obtain the following result.

\begin{thm} \label{thm:kast2}
The partition function \eqref{eq:PF2} of the dimer matching model with constant weights case (i.e. $\alpha_x = \alpha, \beta_x = \beta, \gamma_x = \gamma$ for all $x \in \mathbb{T}_{m_1,m_2}$) is given by $Y_{m_1,m_2}(A) = \sum_{\theta \in \{0,1\}^2}  Z_{m,\theta}(A)$ where
\begin{align*}
Y_{m_1,m_2}(\alpha,\beta,\gamma)  = \frac{1}{2} (-1)^{(\theta_1+m_1+1)(\theta_2+m_2+1)} \prod_{ j \in \mathbb{T}_{m_1,m_2}} \left( \alpha + \beta \exp \left\{ 2 \pi i  \frac{ j_1 + \theta_1 /2}{ m_1 } \right\} + \gamma \exp \left\{ 2 \pi i  \frac{ j_2 + \theta_2 /2}{ m_2 } \right\} \right).
\end{align*}
\end{thm} 

In the constant weights case, provided $\alpha,\beta,\gamma$ satisfy a type of triangle inequality, it is possible to write the partition function as a sum over strictly positive terms. Indeed, the following lemma clarifies the signs the various determinants.

\begin{lemma} \label{lem:signs}
For each $\theta$, if $\alpha,\beta,\gamma$ are real, then $ \det\limits_{ x,y \in \mathbb{T}_{m_1,m_2} } K_m^{A,\theta}(x,y) $ is a real number. More, if $\alpha,\beta,\gamma$ are real and satisfy the triangle inequalities 
\begin{align} \label{eq:triangles}
\alpha \leq \beta + \gamma, \qquad \beta \leq \gamma + \alpha \qquad \text{and} \qquad \gamma \leq \alpha + \beta,  
\end{align}
then the sign of $\det\limits_{ x,y \in \mathbb{T}_{m_1,m_2} } K_m^{A,\theta}(x,y) $ is given by $(-1)^{(\theta_1+m_1+1)(\theta_2+m_2+1)}$. In particular, from Theorem \ref{thm:kast2} we have
\begin{align} \label{eq:nice}
Y_{m_1,m_2}(A) := \frac{1}{2} \sum_{\theta \in \{0,1\}^2}   \prod_{ j \in \mathbb{T}_{m_1,m_2}} \left| \alpha + \beta \exp \left\{ 2 \pi i  \frac{ j_1 + \theta_1 /2}{ m_1 } \right\} + \gamma \exp \left\{ 2 \pi i  \frac{ j_2 + \theta_2 /2}{ m_2 } \right\} \right|.
\end{align}
\end{lemma} 

\begin{proof}
That 
\begin{align*}
\det\limits_{ x,y \in \mathbb{T}_{m_1,m_2} } K_m^{A,\theta}(x,y) = \prod_{ j \in \mathbb{T}_{m_1,m_2}} \left( \alpha + \beta \exp \left\{ 2 \pi i  \frac{ j_1 + \theta_1 /2}{ m_1 } \right\} + \gamma \exp \left\{ 2 \pi i  \frac{ j_2 + \theta_2 /2}{ m_2 } \right\} \right).
\end{align*}
is a real number is a consequence of the fact that is the determinant of a matrix with real entries. Now note that by symmetry, every non-real entry in the product comes in a conjugate pair. In particular, the sign of $\det\limits_{ x,y \in \mathbb{T}_{m_1,m_2} } K_m^{A,\theta}(x,y) $ is $(-1)$ to the power of the number of negative elements occuring in the product over $j \in \mathbb{T}_{m_1,m_2}$. Now note that in order to have a non-paired element lying on the negative real axis, we require 
\begin{align} \label{eq:sign}
 \theta_1 + m_1 = \theta_2 + m_2 = 0 \qquad \mod 2.
\end{align}
In particular, the sign of $\det\limits_{ x,y \in \mathbb{T}_{m_1,m_2} } K_m^{A,\theta}(x,y) $ is given by $(-1)^{(\theta_1 + m_1 + 1)(\theta_2 + m_2 + 1)}$. The result follows. 
\end{proof}

Another important consequence of Lemma \ref{lem:diag} is a derivation of the inverse operator $\left(K^{\alpha,\beta,\gamma}_{m_1,m_2,\theta} \right)^{-1}$ using the eigensystem. Indeed, the reader may verify by direct computation that the inverse of $K^{\alpha,\beta,\gamma}_{m_1,m_2,\theta}$ is given by 
\begin{align} \label{eq:inverse}
\left(K^{\alpha,\beta,\gamma}_{m_1,m_2,\theta}\right)^{-1}(x,y) = \frac{1}{m_1m_2} \sum_{ j \in \mathbb{T}_{m_1,m_2} } \frac{ \exp \left\{ - 2 \pi i \left( \frac{ j_1 + \theta_1 /2}{ m_1 } (y_1-x_1) + \frac{ j_2 + \theta_2/2}{ m_2} (y_2 - x_2) \right) \right\} }{\alpha + \beta \exp \left\{ 2 \pi i  \frac{ j_1 + \theta_1 /2}{ m_1 } \right\} + \gamma \exp \left\{ 2 \pi i  \frac{ j_2 + \theta_2 /2}{ m_2 } \right\} }.
\end{align}
Spelling out our work in this section suppose $\{(x^i,y^i) :1 \leq i \leq p \}$ are a collection of points with $y^i \in \{x^i,x^i+\mathbf{e}_1,x^i+\mathbf{e}_2\}$. For $k \in \{1,2\}$, let $g_k := \{ 1 \leq i \leq j : x^i_k = m_k-1,y^i_k = 0 \}$ be the number of edges $(x^i,y^i)$ that exit and reenter the torus in the $i$-direction. Then according to \eqref{eq:beet2} in the constant weights case, under the signed probability measure $P^{\alpha,\beta,\gamma}_{m_1,m_2,\theta}$, the correlation functions are given by 
\begin{align} \label{eq:beet3}
P^{\alpha,\beta,\gamma}_{m_1,m_2,\theta} \left( \sigma(x^1) = y^1,\ldots,\sigma(x^p) = y^p \right) = & (-1)^{ \theta_1 g_1 + \theta_2 g_2}  \alpha^{ \# \{ i : y^i - x^i = 0 \} }\beta^{ \# \{ i : y^i - x^i = \mathbf{e}_1 \} }   \gamma^{ \# \{ i : y^i - x^i = \mathbf{e}_2 \} } \nonumber \\
& \times   \det \limits_{ i,j = 1}^k \left(  \left(K^{\alpha,\beta,\gamma}_{m_1,m_2,\theta}\right)^{-1} ( y^i,x^j) \right).
\end{align}

Clearly, $\left(K^{\alpha,\beta,\gamma}_{m_1,m_2,\theta}\right)^{-1}:\mathbb{T}_{m_1,m_2} \times \mathbb{T}_{m_1,m_2} \to \mathbb{C}$ is closely related to the operator $L^{\alpha,\beta,\gamma}_{m_1,m_2,\theta}:\mathbb{Z}^2 \to \mathbb{C}$ given by 
\begin{align} \label{eq:inverse2}
L^{\alpha,\beta,\gamma}_{m_1,m_2,\theta}(z) = \frac{1}{m_1m_2} \sum_{ j \in \mathbb{T}_{m_1,m_2} } \frac{ \exp \left\{ - 2 \pi i \left( \frac{ j_1 + \theta_1 /2}{ m_1 } z_1 + \frac{ j_2 + \theta_2/2}{ m_2} z_2 \right) \right\} }{\alpha + \beta \exp \left\{ 2 \pi i  \frac{ j_1 + \theta_1 /2}{ m_1 } \right\} + \gamma \exp \left\{ 2 \pi i  \frac{ j_2 + \theta_2 /2}{ m_2 } \right\} }.
\end{align}

That completes the section on the finite dimer matching model on the discrete torus. 

%%%%%%%%%%%%%%%%%%%%%%%%
\section{Partition function asymptotics: proof of Theorem \ref{thm:pf0}} \label{sec:pf0proof}
%%%%%%%%%%%%%%%%%%%%%%%%

According to Theorem \ref{thm:kast2}, the partition function fo the dimer model in the constant weights case (i.e. $\alpha_x = \alpha, \beta_x = \beta, \gamma_x = \gamma$ for all $x \in \mathbb{T}_{m_1,m_2}$) is given by
\begin{align*}
Z_{m,\theta}(\alpha,\beta,\gamma) = \frac{1}{2} \sum_{\theta \in \{0,1\}^2}  \prod_{ j \in \mathbb{T}_{m_1,m_2}} \left| \alpha + \beta \exp \left\{ 2 \pi i  \frac{ j_1 + \theta_1 /2}{ m_1 } \right\} + \gamma \exp \left\{ 2 \pi i  \frac{ j_2 + \theta_2 /2}{ m_2 } \right\} \right|.
\end{align*}
Now according to \eqref{eq:scapf} we have $Z_n(\lambda,T) = \lim_{m \to \infty} Y_{m,n}\left(1,1-\frac{\lambda}{m},\frac{T}{m} \right)$ where
\begin{align} \label{eq:vesper}
Y_{m,n}\left( 1, 1- \frac{\lambda}{m} , \frac{T}{m} \right) = \frac{1}{2}\sum_{\theta \in \{0,1\}^2}  \prod_{w^n = (-1)^{\theta_2}} \prod_{j=0}^{m-1} \left| 1- \frac{\lambda}{m} + \exp \left\{ 2 \pi i  \frac{ j + \theta_1 /2}{ m } \right\} + \frac{Tw}{m}\right|.
\end{align}
To this end, we note that
\begin{align} \label{eq:ghost}
 \left( 1- \frac{\lambda}{m} + \exp \left\{ 2 \pi i  \frac{ j + \theta_1 /2}{ m } \right\} + \frac{Tw}{m} \right)= e^{ - \lambda/m} \left(  1+ \exp \left\{ 2 \pi i  \frac{ j + \theta_1 /2}{ m } \right\} + \frac{Tw + \lambda}{m} + O(1/m^2) \right).
\end{align}
Plugging \eqref{eq:ghost} into \eqref{eq:vesper}, we see that
\begin{align} \label{eq:vesper2}
\lim_{m \to \infty} Y_{m,n}\left(  1- \frac{\lambda}{m} , 1 , \frac{T}{m} \right) = \frac{e^{-\lambda} }{2}\sum_{\theta \in \{0,1\}^2}  \prod_{w^n = (-1)^{\theta_2}} \lim_{m \to \infty} \prod_{j=0}^{m-1} \left| 1+ \exp \left\{ 2 \pi i  \frac{ j + \theta_1 /2}{ m } \right\} + \frac{Tw + \lambda}{m}\right|.
\end{align}
Studying \eqref{eq:vesper2} we see that our main task in this section is to understand the large-$m$ asymptotics of the product in $m$ terms. 
On this front, we have the following result:

\begin{thm} \label{thm:product max}
For $z \in \mathbb{C}$ and $\theta_1 \in \{0,1\}$ we have
\begin{align} \label{eq:nicemax}
\lim_{m \to \infty, \text{$m$ even} }\prod_{ 0 \leq j \leq m-1}  \left|   1+ e^{  2 \pi i \frac{ j + \theta_1 /2}{m} } + \frac{ z }{m} \right| =  \left| e^z - (-1)^{\theta_1} \right|.
\end{align}
\end{thm}

With a view to proving Theorem \ref{thm:product max}, we note first of all that we may write 
\begin{align*}
\prod_{ 0 \leq j \leq m-1} \left|   1 + e^{  2 \pi i \frac{ j + \theta_1 /2}{m} } + \frac{ z }{m}   \right| = \exp \left\{ A_m(x+iy) + B_m(\theta_1,x+iy)  \right\},
\end{align*}
where
\begin{align} \label{eq:ADEF}
A_m(x+iy)= m \int_0^1 \log \left| 1 +  e^{ 2 \pi i \phi} + \frac{x + iy}{m} \right| \mathrm{d} \phi
\end{align}
is an integral approximating the logarithmic sum and  
\begin{align} \label{eq:BDEF}
B_m(\theta_1,x+iy) = \sum_{ 0 \leq j \leq m-1} \left\{ \log \left| 1+ e^{  2 \pi i \frac{ j + \theta_1/2}{m} } + \frac{ x + iy}{m}  \right| - \int_{-1/2}^{1/2} \log \left| 1+ e^{  2 \pi i \frac{ j + \theta_1/2 + \phi }{m} } + \frac{ x + iy}{m}  \right|  \mathrm{d}\phi    \right\} 
\end{align}
may be thought as an error term for this integral. It transpires that both $A_m$ and $B_m$ contribute asymptotically to the product in Theorem \ref{thm:product max}. We begin by looking at the asymptotics of $A_m$.

\begin{lemma} \label{lem:ALEM}
For $x + iy \in \mathbb{C}$ we have 
\begin{align*}
 \lim_{m \to \infty} A_m(x+iy) = \max \{ x, 0 \}.
\end{align*}
\end{lemma}

\begin{proof}
Write
\begin{align*}
\mathcal{I}(a+bi) := \int_0^1  \mathrm{d} \theta \log \left| a+ib + e^{ 2 \pi i \theta } \right| .
\end{align*}
By symmetry, the reader will note that $\mathcal{I}(a+ib) = \mathcal{I}(|a+ib|)$ for all complex $a+ib$. In particular we may consider $\mathcal{I}(a)$ for positive and real $a$.  In this case we may express $\mathcal{I}(a)$ in terms of the contour integral 
\begin{align*}
\mathcal{I}(a) = \int_{|w|=1} \frac{\mathrm{d}z}{ 2 \pi i z } \log | a + z | = \int_{|w|=1} \frac{\mathrm{d}z}{ 2 \pi i z } \log(  a + z )
\end{align*}
where the latter equality above follows from taking the principal branch $\log : \mathbb{C} \to \mathbb{R} \times i(-\pi,\pi]$ of the complex logarithm and noting that the symmetry about the real axis ensure the imaginary term in the integral is zero.

Now in the setting where $a > 1$, the function $z \mapsto \log(a + z)$ is holomorphic for all $z$ in the unit disc, and it follows from the Cauchy integral formula that
\begin{align*}
\mathcal{I}( a) = \log a \qquad \text{ for $a > 1$. }
\end{align*}
On the other hand, when $a \leq 1$, the integrand is zero. To see this, note that for $a \leq 1$ we have 
\begin{align*}
\mathcal{I}(a) &= \int_{|w| = 1} \frac{ \mathrm{d} w }{ 2 \pi i w} \log(a + w) \\
&=  \int_{|w|=1} \frac{ \mathrm{d} w}{2 \pi i w } \left(  \log(a) + \log(1+w/a) \right)\\
&= \log(a) + \int_{ |w| = 1 } \frac{ \mathrm{d} w}{2 \pi i w }\log(1+ w/ a ) \\
&= \log(a) + \int_{ |w| = 1 } \frac{ \mathrm{d} w}{2 \pi i w } \left( \log(1/w) + \log(w+1/a ) \right) \\
&= \log (a) + 0 + \log(1/a) = 0.
\end{align*}
In particular, for $a+ bi \in \mathbb{C}$, $\mathcal{I}( |a+bi | ) = \max \{ \log|a+bi|,0\}$.

It follows that
\begin{align*}
 m \int_0^1  d \theta \log \left| 1 + e^{ 2 \pi i \theta } +  \frac{ x+ iy}{m}  \right| = \max \{ x, 0 \} = m \max \left\{ \log| 1 + \frac{ x + iy }{ m } | , 0 \right\}.
\end{align*}
The result in question follows from noting that $| 1 + \frac{ x + iy }{ m } | = 1 + \frac{x}{m} + o(1/m)$, and hence $m \log  |1 + \frac{ x+ iy}{m }| = m\log(1 + \frac{x}{m} + o(1/m) ) = x + o(1)$. 
\end{proof}

We now turn to understand the asymptotic behaviour of $B_m$. 

\begin{lemma} \label{lem:tech}
With $B_m(\theta_1,x+iy)$ as in \eqref{eq:BDEF} we have $\lim_{m \to \infty, \text{$m$ even}} B_m(\theta_1,x+iy) = C(\theta_1,x + iy )$ where 
\begin{align*}
C(\theta_1,x + iy) :=  \frac{1}{2} \sum_{ k \in \mathbb{Z}} \int_{-1/2}^{1/2} \left\{ \log\left( 1 + \left( \frac{ y/2\pi + \theta_1/2 + k }{ x/2 \pi  }  \right)^2 \right) - \log\left( 1  + \left(\frac{ y/2\pi + \theta_1/2 + k + \phi }{ x/2\pi}   \right)^2  \mathrm{d} \phi  \right)  
 \right\}.
\end{align*}
\end{lemma}

\begin{proof}
We will show that $B_m(\theta_1,x+iy) = C(\theta_1,x+iy) + O(m^{-1/4})$, where the $O$ term depends on $x$ and $y$.

Recall that $m$ is even. Consider that we can rewrite the $m$ roots of $(-1)^{\theta_1}$ as $\{ e^{ 2 \pi i \frac{ j + \theta_1 /2}{ m } } : 0 \leq j \leq m-1 \} =  \{ e^{ 2 \pi i \frac{ j + \theta_1/2}{ m } } : - ( \frac{m}{2} -1) \leq j \leq \frac{m}{2} \}$. It follows that we can write 
\begin{align*}
B_m &= \sum_{ - \lfloor m/2 \rfloor  \leq j \leq \lfloor \frac{m-1}{2} \rfloor   } \left\{ \log \left| 1- e^{  2 \pi i \frac{ j + \theta_1/2}{m} } + \frac{ x + iy}{m}  \right| - \int_{-1/2}^{1/2} \log \left| 1 - e^{  2 \pi i \frac{ j + \phi+   I /2}{m} } + \frac{ x + iy}{m}  \right|      \right\} \\
&=: - \sum_{ - \lfloor m/2 \rfloor  \leq j \leq \lfloor \frac{m-1}{2} \rfloor   } b_j,
\end{align*}
where in the former line, $I := \theta + m ~ \mathrm{mod} ~ 2$, and the summands are given by
\begin{align*}
b_j := \mathrm{Re} \int_{ -1/2}^{1/2} \log \left( 1 - \frac{ e^{ 2 \pi i \frac{j + \theta_1/2}{ m} } \left( e^{ 2 \pi i \phi/m} -1 \right)}{ 1- e^{  2 \pi i \frac{ j + \theta_1/2}{m} } + \frac{ x + iy}{m}  } \right) \mathrm{d} \phi.
\end{align*}
We begin by controlling the sizes of $b_j$. To this end, first we note that there is a constant $c = c(x,y,I) > 0$ not depending on $m$ such that for all $- \lfloor m/2 \rfloor  \leq j \leq \lfloor \frac{m-1}{2} \rfloor $ we have 
\begin{align*}
\left| 1- e^{  2 \pi i \frac{ j + \theta_1/2}{m} } + \frac{ x + iy}{m}  \right| \geq c j/m.
\end{align*}
In particular, using the fact that $|e^{ 2 \pi i \phi /m } - 1  | \leq 2 \pi \phi/m$ we have
\begin{align*}
\left| \frac{ e^{ 2 \pi i \frac{j + \theta_1/2}{ m} } \left( e^{ 2 \pi i \phi/m} -1 \right)}{ 1- e^{  2 \pi i \frac{ j + \theta_1/2}{m} } + \frac{ x + iy}{m}  } \right| \leq \frac{ 2 \pi }{ c j}.
\end{align*} 
Let $j_0$ be chosen such that $\frac{2 \pi }{ c j_0} \leq 1/2$. Then uniformly for $|x| \leq 1/2$ we have $\log(1 + x) = x + O(x^2)$ so that:
\begin{align*}
\log \left( 1 - \frac{ e^{ 2 \pi i \frac{j + \theta_1/2}{ m} } \left( e^{ 2 \pi i \phi/m} -1 \right)}{ 1- e^{  2 \pi i \frac{ j + \theta_1/2}{m} } + \frac{ x + iy}{m}  } \right) &=- \frac{ e^{ 2 \pi i \frac{j + \theta_1/2}{ m} } \left( e^{ 2 \pi i \phi/m} -1 \right)}{ 1- e^{  2 \pi i \frac{ j + \theta_1/2}{m} } + \frac{ x + iy}{m}  } + O (1/j^2)\\
 &=- \frac{ e^{ 2 \pi i \frac{j + \theta_1/2}{ m} } \left( 2 \pi i \phi/m \right)}{ 1- e^{  2 \pi i \frac{ j + \theta_1/2}{m} } + \frac{ x + iy}{m}  } + O \left( \frac{ (\phi/m)^2}{c j / m } \right) + O (1/j^2),
\end{align*}
where the bounds are uniform in $\phi \in [-1/2,1/2]$. In particular, 
\begin{align*}
\int_{-1/2}^{1/2} \log \left( 1 - \frac{ e^{ 2 \pi i \frac{j + \theta_1/2}{ m} } \left( e^{ 2 \pi i \phi/m} -1 \right)}{ 1- e^{  2 \pi i \frac{ j + \theta_1/2}{m} } + \frac{ x + iy}{m}  } \right)  \mathrm{d} \phi = O(1/j^2).
\end{align*}
In particular, the contribution to $B_m$ from $|j| \geq m^{1/4}$ is asymptotically negligable. More specifically,
\begin{align*}
 \sum_{ -\lfloor m/2 \rfloor \leq j \leq - m^{1/4}} b_j + \sum_{ m^{1/4} \leq j \leq \lfloor \frac{m-1}{2} \rfloor }  b_j =  O \left( \sum_{ j \geq m^{1/4}} 1/j^2 \right) = O(m^{-1/4}).
\end{align*}
We turn to studying the contribution from $|j| \leq m^{1/4}$. 
A straightforward computation tells us that uniformly for all $j,m$ with  $|j| \leq m^{1/4}$ and for all $\phi \in [-1/2,1/2]$ we have
\begin{align*}
m \left( 1 - e^{  2 \pi i \frac{ j + \phi+   I /2}{m} } + \frac{ x + iy}{m}  \right) = x + (-y + 2\pi( j + \theta_1/2 + \phi) )i + O(m^{-1/2}).
\end{align*} 

In particular,
\begin{align*}
& \log \left| 1+ e^{  2 \pi i \frac{ j + \theta_1/2}{m} } + \frac{ x + iy}{m}  \right| - \int_{-1/2}^{1/2} \log \left| 1+ e^{  2 \pi i \frac{ j + \phi+   \theta_1/2}{m} } + \frac{ x + iy}{m}  \right|\\
&=  \log \left|  x + (-y + 2 \pi (j + \theta_1/2 ) )i    \right| - \int_{-1/2}^{1/2} \log \left|  x + (-y + 2 \pi (j + \theta_1/2 + \phi ) )i   \right|        \mathrm{d}\phi + O(m^{-1/2})\\
&=: c_j + O(m^{-1/2}).
\end{align*}
It follows that 
\begin{align*}
\sum_{ |j| \leq m^{1/4} } b_j = \sum_{ |j|  \leq m^{1/4}} c_j  + O( m^{-1/4} ),
\end{align*}
where the $O(m^{-1/4})$ term comes from the summation of $2m^{1/4}$ summands of order $O(m^{-1/2})$. 

Finally, we note that $c_{-j}$ may be rewritten 
\begin{align*}
c_{-j} := - \frac{1}{2} \int_{-1/2}^{1/2} \left\{ \log\left( 1 + \left( \frac{ y/2\pi + \theta_1/2 + j }{ x/2 \pi }  \right)^2 \right) - \log\left( 1  + \left(\frac{ y/2\pi + \theta_1/2 + j + \phi }{ x/ 2 \pi }   \right)^2 \right)  
 \right\} \mathrm{d} \phi,
\end{align*}
and consequently, it's easily shown that $c_j = O(1/j^2)$. In particular, $\sum_{j \in \mathbb{Z}} |c_j| < \infty$, and 
\begin{align*}
\sum_{ |j|  \leq m^{1/4}} c_j   = \sum_{ j \in \mathbb{Z}} c_j + O( m^{ - 1/4} ),
\end{align*}
completing the result.

\end{proof}

It transpires that the sum occuring in Lemma \ref{lem:tech} is explicitly computable. In this direction, we begin with the following lemma.
\begin{lemma}
For $q > 0$ and $p \in [0,1]$ we have
\begin{align*}
 \frac{1}{2} \sum_{ j \in \mathbb{Z}} \log \left( 1 + \left( \frac{q}{p+j}\right)^2 \right) = \log | \sin( \pi(p+qi)| - \log | \sin(\pi p ) |.
\end{align*}
\end{lemma}
\begin{proof}
We begin with the following product representation of the gamma function:
\begin{align*}
\left| \frac{ \Gamma (p) }{ \Gamma( p + qi) } \right|^2 := \prod_{ k = 0}^\infty \left( 1 + \left( \frac{q}{p+k} \right)^2 \right).
\end{align*}
In particular,
\begin{align*}
\prod_{ k \in \mathbb{Z}} \left( 1 + \left( \frac{q}{p+k} \right)^2 \right) = \frac{1}{ 1 + (q/p)^2 } \left| \frac{ \Gamma (p) }{ \Gamma( p + qi) } \right|^2 \left| \frac{ \Gamma (-p) }{ \Gamma( -p + qi) } \right|^2 
\end{align*}
Using the property $\bar{\Gamma(z)} = \Gamma( \bar{z})$, and then the additional identity $\Gamma(z) \Gamma(-z) = - \frac{ \pi}{ z \sin( \pi z) }$ we have 
\begin{align*}
\prod_{ k \in \mathbb{Z}} \left( 1 + \left( \frac{q}{p+k} \right)^2 \right) = \left| \frac{ \sin( \pi( p + qi ) }{ \sin( \pi p ) }  \right|^2.
\end{align*}
Now take logarithms.
\end{proof}

\begin{lemma} \label{lem:sum int}
For $q >0$ and $p \in [0,1]$ define
\begin{align*}
F(p, q) := \frac{1}{2} \sum_{ k \in \mathbb{Z}} \int_{-1/2}^{1/2} \left\{ \log\left( 1  + \left(
\frac{ p + k }{q}  \right)^2 \right) - \log\left( 1  + \left( \frac{ p + k + \phi }{ q}  \right)^2 \right)  
 \right\} \mathrm{d} \phi.
\end{align*}
Then $F(p,q) = \log \left| 2 \sin ( \pi ( p + q i ) \right|  - \pi q$. 
\begin{proof}
Clearly $F(p,q)$ may be written $\lim\limits_{N \to \infty} F_N(p,q)$ where
\begin{align} \label{eq:decomp}
F_N(p,q) := \frac{1}{2} \sum_{ - N \leq k \leq N }  \log\left( 1  + \left(
\frac{ p + k }{q}  \right)^2 \right) - \frac{1}{2} \int_{ - N - 1/2 }^{N + 1/2 }  \log\left( 1  + \left( \frac{ p + k + \phi }{ q}  \right)^2 \right)  .
\mathrm{d} \phi 
\end{align}
Now on the one hand, using the fact that $\frac{ \mathrm{d}}{ \mathrm{d} x} \left( x \log (1 + x^2 ) - 2x + 2\tan^{-1}(x) \right) = \log(1 + x^2)$ to obtain the first inequality below, and simple properties of the logarithm function to obtian the second, we have 
\begin{align*}
&\int_{ - N - 1/2 }^{N + 1/2 } \log\left( 1  + \left( \frac{ p + k + \phi }{ q}  \right)^2 \right) \mathrm{d}\phi  \\
&= \frac{q}{2} \left\{  \frac{N+1/2+p}{q} \log \left( 1 + \left( \frac{N+1/2+p}{q} \right)^2 \right) -  \frac{N+1/2-p}{q} \log \left( 1 + \left( \frac{N+1/2-p}{q} \right)^2 \right)  \right\}\\
& - (2N + 1) +  q \left( \tan^{-1}(N+1/2+p ) - \tan^{-1}( - N -1/2 + p) \right)\\
&= (2N+1)\log N - (\log (q) + 1) (2 N + 1) +1 + q\pi + o(1).
\end{align*}
On the other hand, we turn to looking at the sum in \eqref{eq:decomp}, which may be expanded as 
\begin{align*}
\frac{1}{2} \sum_{ - N \leq k \leq N }  \log\left( 1  + \left(
\frac{ p + k }{q}  \right)^2 \right) = \sum_{ -N \leq j \leq N } \log | p + j| -(2N+1) \log q + \frac{1}{2} \sum_{ -N \leq j \leq N } \log \left(  1 + \left( \frac{ q}{ p + j} \right)^2 \right).
\end{align*} 
Using the definition of the gamma function, we clearly have $\prod_{ j = 1}^N |j \pm p| = \frac{ \Gamma(N+1 \pm p)}{ \Gamma(1 \pm p)}$. Using this fact to obtain the first inequality below, and then Stirling's asymptotics for the gamma function to obtain the second, we have 
\begin{align*}
\frac{1}{2} \sum_{ - N \leq k \leq N }  \log\left( 1  + \left(
\frac{ p + k }{q}  \right)^2 \right) &= \log \Gamma(N+1+p) + \log \Gamma(N+1 - p) - \log \Gamma(1+p) - \log \Gamma(1-p) \\
&+ \log p - (2N+1) \log (q) +  \frac{1}{2} \sum_{ -N \leq j \leq N } \log \left(  1 + \left( \frac{ q}{ p + j} \right)^2 \right)\\
&= (2N+1) \log N - (\log(q) + 1)(2N+1)\\
&+\log(p) +1 -\log \Gamma(1-p) - \log\Gamma(1+p) + \log(2\pi)\\
& + \log | \sin( \pi(p+qi)| - \log | \sin(\pi p ) | + o(1).
\end{align*} 
In particular, using \eqref{eq:decomp} we have 
\begin{align*}
F(p,q) = \log(p) -\log \Gamma(1-p) - \log\Gamma(1+p) + \log(2\pi) + \log | \sin( \pi(p+qi)| - \log | \sin(\pi p ) | - \pi q.
\end{align*}
Now by the functional equation $\Gamma( 1 - z) \Gamma(z) = \pi / \sin (\pi z)$, as well as the fact that $\Gamma(1 + z) = z \Gamma(z)$. This simplifies further still to
\begin{align*}
F(p,q) = \log \left| 2 \sin ( \pi ( p + q i ) \right|  - \pi q.
\end{align*} 
\end{proof}

\end{lemma}

We are now equipped to prove Theorem \ref{thm:product max}.

\begin{proof}[Proof of Theorem \ref{thm:product max}]
Recall that we wrote 
\begin{align} \label{eq:0eq}
\left| \prod_{ 0 \leq j \leq m-1} \left(  1+ e^{  2 \pi i \frac{ j + \theta /2}{m} } + \frac{ x + iy}{m}   \right) \right|   = e^{A_m + B_m}
\end{align}
 where $A_m$ and $B_m$ are given in \eqref{eq:ADEF} and \eqref{eq:BDEF} respectively. On the one hand, thanks to Lemma \ref{lem:ALEM}, 
\begin{align} \label{eq:Aeq}
\lim_{m \to \infty} A_m( x + iy) = \max \{ x, 0 \}
\end{align}
On the other hand, with $F(p,q)$ as in the statement of Lemma \ref{lem:sum int}, using Lemma \ref{lem:tech} with Lemma \ref{lem:sum int} we have
\begin{align*}
\lim_{m \to \infty} B_m ( \theta_1,x + iy) = C(\theta_1,x + i y) = F( \theta_1/2 + y / 2 \pi , x / 2 \pi ),
\end{align*}
by virtue of Lemma \ref{lem:sum int} is equal to 
\begin{align} \label{eq:Beq}
\lim_{m \to \infty} B_m (\theta_1, x + iy)  := \log \left| 2 \sin \left( \frac{ \pi \theta_1}{2}  +  \frac{y}{2} + \frac{ x i}{2}   \right) \right| - \frac{1}{2} |x|. 
\end{align}
By combining \eqref{eq:Aeq} and \eqref{eq:Beq} in \eqref{eq:0eq}, and noting that $\max\{ x,0\} - \frac{1}{2}|x| = \frac{1}{2}x$, we see that 
\begin{align} \label{eq:sinrep}
\lim_{m \to \infty} \left| \prod_{ 0 \leq j \leq m-1} \left(  1+ e^{  2 \pi i \frac{ j + \theta /2}{m} } + \frac{ x + iy}{m}   \right) \right| = 2 e^{x/2} \left| \sin \left( \frac{ \pi \theta_1}{2}  +  \frac{y}{2} + \frac{ x i}{2}   \right) \right|.
\end{align}
Making use of the representation $\sin(z) = \frac{1}{2i}(e^{ i z} - e^{ - iz})$, and further using the identity $\sin(\bar{z}) = \overline{ \sin(z)}$, it is possible to show further as a consequence of \eqref{eq:sinrep} that
if $z = x+iy$ then
\begin{align*}
2 e^{x/2} \left| \sin \left( \frac{ \pi \theta_1}{2}  +  \frac{y}{2} + \frac{ x i}{2}   \right) \right| = \left| e^z - (-1)^{\theta_1} \right|
\end{align*}
from which we obtain Theorem \ref{thm:product max}.
\end{proof}

We are now ready to prove Theorem \ref{thm:pf0}.

\begin{proof}[Proof of Theorem \ref{thm:pf0}]
According to \eqref{eq:scapf}, $Z_n(\lambda,T) = \lim_{m \to \infty} K_{m,n}(1,1-\frac{\lambda}{m},\frac{T}{m})$. Plugging Theorem \ref{thm:product max} into \eqref{eq:vesper2}, we see that
\begin{align*}
Z_n(\lambda,T) &:= \lim_{ m \to \infty, \text{m even}}  Z_{m,n}\left(  1- \frac{\lambda}{m} , 1 , \frac{T}{m} \right) = \frac{e^{-\lambda}}{2}\sum_{\theta \in \{0,1\}^2}  \prod_{w^n = (-1)^{\theta_2}} \left| e^{ Tw + \lambda} - (-1)^{\theta_1} \right|.
\end{align*}
Using Theorem \ref{thm:product max} with $z = Tw - \lambda$ and \eqref{eq:vesper} we have
\begin{align*}
Z_n(\lambda,T) = \frac{1}{2}\sum_{\theta \in \{0,1\}^2}  \prod_{w^n = (-1)^{\theta_2}} \left| e^{ Tw} - (-1)^{\theta_1} e^{ - \lambda} \right|.
\end{align*}
For $T,\lambda$ real, using an argument similar to that in the proof of Lemma \ref{lem:signs}, the sign of $\prod_{w^n = (-1)^{\theta_2}}( e^{ Tw } - (-1)^{\theta_1} e^{- \lambda} )$ is given by $(-1)^{(\theta_1+1)(\theta_2+n+1)}$. Theorem \ref{thm:pf0} follows.
\end{proof}

%%%%%%%%%%%%%%%%%%%%%%%%
\section{Correlation function asymptotics: proof of Theorem \ref{thm:corr}} \label{sec:corrproof}
%%%%%%%%%%%%%%%%%%%%%%%%

\subsection{Overview}
Recall from the introduction that we explained how every dimer matching $\sigma:\mathbb{T}_{m,n} \to \mathbb{T}_{m,n}$ gives rise to a generalised bead configuration on $\mathbb{T}_n := [0,1) \times \mathbb{Z}_n$ in which some of the inequalities in \eqref{eq:interlace} may not be strict. 

Recall further that for general weights $(\alpha_x,\beta_x,\gamma_x)$ at the end of Section \ref{sec:basic} we defined a signed measure $P_{A,\theta}$ on the set of dimer matchings on $\mathbb{T}_{m,n}$, which naturally gives rise to a generalised bead configuration on $\mathbb{T}_n$. Under the scaling limit set out in \eqref{eq:scaling} we have the convergence in distribution
\begin{align*}
P^{1,1-\frac{\lambda}{m},\frac{T}{m}}_{m,n} \to \mathbf{P}^{\lambda,\theta,T} \qquad \text{as $m \to \infty$}
\end{align*}
where $\mathbf{P}_n^{\alpha,\theta,T}$ is the measure defined in the introduction. 

According to \eqref{eq:beet3} and \eqref{eq:inverse}, under the signed probability measure $P^{1,1-\frac{\lambda}{m},\frac{T}{m}}_{m,n}$ 
the correlation functions are given by 
\begin{align} \label{eq:beet44}
P^{1,1-\frac{\lambda}{m},\frac{T}{m}}_{m,n} \left( \sigma(x^1) = y^1,\ldots,\sigma(x^p) = y^p \right) = & (-1)^{ \theta_1 g_1 + \theta_2 g_2}  \left( 1-  \frac{T}{m} \right)^{ \# \{ i : y^i - x^i = \mathbf{e}_1 \} }   \left( \frac{T}{m} \right)^{ \# \{ i : y^i - x^i = \mathbf{e}_2 \} } \nonumber \\
& \times   \det \limits_{ i,j = 1}^k \left(  \left(K^{\alpha,\beta,\gamma}_{m_1,m_2,\theta}\right)^{-1} ( y^i,x^j) \right),
\end{align}
where 
\begin{align} \label{eq:inverse44}
\left(K^{1,1-\frac{\lambda}{m},\frac{T}{m}}_{m_1,m_2,\theta}\right)^{-1}(x,y) = \frac{1}{m_1m_2} \sum_{ j \in \mathbb{T}_{m_1,m_2} } \frac{ \exp \left\{ - 2 \pi i \left( \frac{ j_1 + \theta_1 /2}{ m_1 } (y_1-x_1) + \frac{ j_2 + \theta_2/2}{ m_2} (y_2 - x_2) \right) \right\} }{1 + \left( 1 - \frac{\lambda}{m} \right)  \exp \left\{ 2 \pi i  \frac{ j_1 + \theta_1 /2}{ m_1 } \right\} + \frac{T}{m} \exp \left\{ 2 \pi i  \frac{ j_2 + \theta_2 /2}{ m_2 } \right\} },
\end{align}
and $m_1 = m$ and $m_2 = n$. Separating out the dependence on $m$, this may be rewritten
\begin{align} \label{eq:inverse45}
\left(K^{1,1-\frac{\lambda}{m},\frac{T}{m}}_{m,n,\theta}\right)^{-1}(x,y) = \frac{1}{n} \sum_{ w^n = (-1)^{\theta_2}} w^{ - (y_2 - x_2) } \frac{1}{m} \sum_{ j = 0}^{m-1}  \frac{ \exp \left\{ - 2 \pi i \frac{ j_1 + \theta_1 /2}{ m} (y_1-x_1) \right\}  }{ 1 + \left( 1 - \frac{\lambda}{m} \right)  \exp \left\{ 2 \pi i  \frac{ j_1 + \theta_1 /2}{ m_1 } \right\} + \frac{Tw}{m}  }.
\end{align}
Now given a point $y$ in $\mathbb{T}_n := [0,1) \times \mathbb{Z}_n$, we define a scaled point $\tilde{y}$ in $\mathbb{T}_{m,n} = \mathbb{Z}_{m} \times \mathbb{Z}_{n}$ by setting
\begin{align*}
x = (t,h) \implies \tilde{x} = ([[tm]],h),
\end{align*}
where $[[tm]]$ is the largest \emph{even} integer less than $tm$. The convention of taking this integer to be even substantially expedites our discussion later on.

It follows using the association set out in the introduction that that
\begin{align*}
\mathbf{P}_n^{\alpha,\theta,T} ( \Gamma(x_i : i \in \mathcal{B} \sqcup \mathcal{O} \sqcup \mathcal{U} )) = \lim_{m \to \infty} m^{ \# \mathcal{B} }  P^{1,1-\frac{\lambda}{m},\frac{T}{m}}_{m,n} \left( \sigma(\tilde{x}^1) = \tilde{y}^1,\ldots,\sigma(\tilde{x}^p) = \tilde{y}^p \right) ,
\end{align*}
where
\begin{align*}
\tilde{y}^i := \tilde{x}^i + \mathrm{1}_{i \in \mathcal{O}} \mathbf{e}_1 + \mathrm{1}_{i \in \mathcal{B}} \mathbf{e}_2
\end{align*}
As such, our main task in this section is to undertake a careful analysis of the asymptotics of these inverse operators $(K^{1,1-\frac{\lambda}{m},\frac{T}{m}}_{m_1,m_2,\theta})^{-1}(x,y)$ as $m \to \infty$, where the horizontal distance between $x$ and $y$ may be of order $m$. It turns out that this scaling limit is quite delicate. Indeed, remarkably perhaps, for fixed $x$, the $(K^{1,1-\frac{\lambda}{m},\frac{T}{m}}_{m,n,\theta})^{-1}(x,x)$ may behave quite differently to $(K^{1,1-\frac{\lambda}{m},\frac{T}{m}}_{m,n,\theta})^{-1}(x+\mathbf{e}_1,x)  $ as $m \to \infty$. As such, it is necessary to be pedantic about the roles of the 
indicator functions $\mathrm{1}_{i \in \mathcal{O}} \mathbf{e}_1 + \mathrm{1}_{i \in \mathcal{B}} \mathbf{e}_2$ in the asymptotics.

In any case, the primary computational result of this section, allowing us to handle the large $m$ asymptotics of the terms in \eqref{eq:inverse45}, is the following.

\begin{thm} \label{thm:inverselim}
Let $s \in (-1,1)$ and $z \in \mathbb{C}$ such that $2\pi i j + z$ is nonzero for all $j \in \mathbb{Z}$. Then
\begin{align} \label{eq:inverselim}
\lim_{ m \to \infty, m \text{ even}} (-1)^{ \lfloor sm \rfloor} \frac{1}{m}  \sum_{ j = 0}^{m-1} \frac{ e^{  - 2 \pi i \frac{ j + \theta_1/2}{m} \lfloor sm \rfloor } }{ 1 +  \exp \left\{ 2 \pi i  \frac{ j + \theta_1 /2}{ m } \right\} + \frac{ z}{m} } =  (-1)^{ \theta_1 \mathbf{1}_{ s < 0 } } \frac{ e^{ - z (s+\mathbf{1}_{s < 0}) } }{ 1 - (-1)^{\theta_1} e^{-z} }.
\end{align}
We also have the following equation (analogous to replacing $\lfloor sm \rfloor $ with $-1$), 
\begin{align} \label{eq:inverselim0}
\lim_{ m \to \infty, m \text{ even}}  \frac{1}{m} \sum_{ j = 0}^{m-1} \frac{ e^{   2 \pi i \frac{ j + \theta_1/2}{m} }  }{ 1 +  \exp \left\{ 2 \pi i  \frac{ j + \theta_1 /2}{ m } \right\} + \frac{ z}{m} } =  (-1)^{ \theta_1 +1 } \frac{ e^{ - z } }{ 1 - (-1)^{\theta_1} e^{-z} }.
\end{align}
\end{thm}

After proving Theorem \ref{thm:inverselim}, we use it to prove Theorem \ref{thm:corr}.

\subsection{Proof of Theorem \ref{thm:inverselim}}  \label{sec:inverselimproof}

Let us begin by noting that while the first equation \eqref{eq:inverselim} of Theorem \ref{thm:inverselim} is claimed for all  $s \in (-1,1)$, we may assume without loss of generality that $s \in [0,1)$. To see this, consider that if $\tilde{s} := s+\mathbf{1}_{s < 0}$, then $\lfloor \tilde{s} m \rfloor = \lfloor sm \rfloor + m \mathbf{1}_{s < 0}$. It now remains to note that for any $0 \leq j \leq m-1$ we have
\begin{align*}
 e^{  - 2 \pi i \frac{ j + \theta_1/2}{m} \lfloor sm \rfloor }  = (-1)^{\theta_1 \mathbf{1}_{s < 0}}  e^{  - 2 \pi i \frac{ j + \theta_1/2}{m} \lfloor \tilde{s} m \rfloor } .
\end{align*}
We thus assume for the remainder of the proof of Theorem \ref{thm:inverselim} that $s \geq 0$. (We will treat the case of equation \eqref{eq:inverselim0} separately at the end.)

We take a moment to overview the proof of Theorem \ref{thm:inverselim}. 
Let us note that since $m$ is even we may reindex the roots of $(-1)^{\theta_1}$ as follows:
\begin{align*}
\left\{ e^{ 2 \pi i \frac{ j + \theta_1/2}{m}} : 0 \leq j \leq m \right\} = \left\{ -  e^{ - 2 \pi i \frac{ j + \theta_1/2}{m} } : -m/2+1\leq j \leq m/2 \right\}.
\end{align*} 
Using this reindexing to study the quantity in Theorem \ref{thm:inverselim}, we have  
\begin{align} \label{eq:lion}
\frac{1}{m} \sum_{ j = 0}^{m-1} \frac{ e^{ - 2 \pi i \frac{ j + \theta_1/2}{m} \lfloor sm \rfloor } }{ 1  +  \exp \left\{ 2 \pi i  \frac{ j + \theta_1 /2}{ m } \right\} + \frac{z}{m}  } = (-1)^{\lfloor sm \rfloor} \frac{1}{m} \sum_{ j = - m/2+1}^{ m/2}  \frac{ e^{ 2 \pi i \frac{ j + \theta_1/2}{m} \lfloor sm \rfloor } }{1 -  \exp \left\{- 2 \pi i  \frac{ j + \theta_1/2}{m} \right\} +  \frac{ z}{m} }.
\end{align}
Consider that for large $m$ and $j$ small compared to $m$, the quantity occuring in the denominator of the summands takes the form
\begin{align*}
m \left( 1 -  \exp \left\{ - 2 \pi i  \frac{ j + \theta_1/2}{m} \right\} +  \frac{ z}{m} \right) =  2 \pi i (j + \theta_1/2) + z + o(1/m). 
\end{align*}
One then might expect, naively taking $m$ to infinity with the quantity on the right-hand-side  \eqref{eq:lion}, that we have 
\begin{align} \label{eq:puresum}
\lim_{m \to \infty} \frac{1}{m} \sum_{ j = - m/2 + 1}^{ m/2 }  \frac{ e^{   2 \pi i \frac{ j + \theta_1/2}{m} \lfloor sm \rfloor } }{1 -  \exp \left\{- 2 \pi i  \frac{ j + \theta_1/2}{m} \right\} +  \frac{ z }{m} } =^{?} e^{ \pi \theta_1 s} \sum_{ j \in \mathbb{Z}} \frac{ e^{ 2 \pi i j s }}{ 2 \pi i(j + \theta_1/2)  + z} ,
\end{align}
insofar as it is possible to make sense of the complex sum on the right-hand-side of \eqref{eq:puresum}, which is clearly not uniformly convergent. It turns that \eqref{eq:puresum} is only true when $s \in (0,1)$; in fact, when $s =0$ an extra term emerges due to a contour integral. Thus our task in the section is twofold, in main we have the technical task of establishing the validity of the \eqref{eq:puresum}, and thereafter we use Fourier analysis to evaluate the sums of the form occuring on the right.

More specifically, the remainder of this section is structured as follows:
\begin{itemize}
\item In Lemma \ref{lem:tech1} we show that while the expression
\begin{align*}
F(z,s) := \sum_{ j \in \mathbb{Z}} \frac{e^{2 \pi i j s }}{ 2 \pi i j + z } \qquad s \in [0,1), z \in \mathbb{C} - 2 \pi i \mathbb{Z}, 
\end{align*}
does not represent a uniformly convergent sum, that the partial sums over $|j| \leq n$ are Cauchy, and hence the sum over $j \in \mathbb{Z}$ makes sense as a limit of these partial sums.
\item In Lemma \ref{lem:fourier}, we use the Poisson summation formula to compute $F(z,s) = e^{ - zs}(1-e^{-z})^{-1}$ for $s \in (0,1)$, and establish a related expression in the case $s = 0$. 
\item In Lemma \ref{lem:techs} we study show that an equation of the form \eqref{eq:puresum} holds when $s > 0$, and use the previous Lemma \ref{lem:fourier} to evaluate the expression on the right-hand-side of \eqref{eq:puresum} in this case. 
\item In Lemma \ref{lem:tech0} we see that the logic of \eqref{eq:puresum} fails when $s = 0$, since there is an asymptotic contribution from large $j$. Here we evaluate the sum component (i.e. contribution from small $j$) using Lemma \ref{lem:fourier}, and show that the asymptotic contribution from large $j$ may be computed by way of a contour integral. 
\item Finally, Lemmas \ref{lem:techs} and \ref{lem:tech0} are quickly tied together to establish Theorem \ref{thm:inverselim}.
\end{itemize}

We begin with Lemma \ref{lem:tech1}, which states that sums of the form $\sum_{ k \in \mathbb{Z}} \frac{e^{ 2 \pi i k s}}{ 2 \pi i k + z}$, while clearly not absolutely convergent, do indeed make sense.

\begin{lemma} \label{lem:tech1}
Let $s \in [0,1)$ and let $z \in \mathbb{C}$ such that the $z +2 \pi i j $ is nonzero for every integer $j$. Then there is a constant $C(s) > 0$ independent of $n_1, n_2$ and $z$ such that 
\begin{align} \label{eq:boundc}
\left| \sum_{ n_1 < |j| \leq n_2 } \frac{ e^{ 2 \pi i j s } }{ 2 \pi i j + z }  \right| \leq \frac{C(s)}{ n_1 - |z| }
\end{align}
In particular, for all $s \in [0,1)$ the sequence $F_{n}(z,s) := \sum_{ |j| \leq n } \frac{ e^{ 2 \pi i j s } }{ 2 \pi i j + z }$ is Cauchy, and hence the limit
\begin{align*}
F(z,s) :=\lim_{n \to \infty} F_n(z,s) = \sum_{ j \in \mathbb{Z}} \frac{ e^{ 2 \pi i j s } }{ 2 \pi i j + z } 
\end{align*}
exists. 
\end{lemma}

\begin{proof}
The case $s = 0$ is straightforward: we may couple the $j$ and $-j$ terms, and use the identity
\begin{align*}
\frac{1}{2 \pi i j  + z } + \frac{1}{-2\pi i j + z} = \frac{ 2z}{ (2 \pi j)^2 + z^2 }
\end{align*}
in order to establish that \eqref{eq:boundc} holds in the case $s = 0$ for some $C(0)$.

We turn to the case where $s \in (0,1)$. Here, let $b = b(s) := e^{  2 \pi i s}$ and $a_j = (2 \pi ji  + z )^{-1}$. A brief calculation tells us that there is an absolute constant $C$ such that setting $c_j := a_j - a_{j+1}$ we have
\begin{align} \label{eq:gorilla}
|c_j| < \frac{ C}{  (|j| - |z|)^2 }.
\end{align}
Note that the sum $a_j = \sum_{ \ell \geq j} c_\ell$ \emph{is} uniformly convergent. Now we can write 
\begin{align} \label{eq:pear}
\sum_{ n_1 < j \leq n_2 } a_j b^j  = \sum_{ n_1 < j \leq n_2 } b^j \sum_{ \ell \geq j} c_j = \sum_{ n_1 < \ell \leq n_2 } c_\ell \sum_{ n_1 < j \leq \ell }b^j.
\end{align}
Now note that since $s \in (0,1)$, $b(s) \neq 0$. Since there is a constant $C$ such that $\left| \sum_{ j = 0}^{m-1} b^j \right| \leq 2/|1-b(s)| \leq C(s)$ for every $m > 0$, it follows from \eqref{eq:pear} and \eqref{eq:gorilla} we have
\begin{align*}
\left| \sum_{ n_1 < j \leq n_2 } a_j b^j   \right| \leq C(s) \sum_{ \ell \geq n_1 } | c_\ell | \leq \frac{ C_1(s) }{ |j| - |z| } ,
\end{align*} 
for a second sufficiently large constant $C_1(s)$, completing the proof.
\end{proof}

Our next result states that $F(z,s)$ in Lemma \ref{lem:tech1} may be computed explicitly. 

\begin{lemma} \label{lem:fourier}
Let $s \in (0,1)$ and let $z \in \mathbb{C}$. Then 
\begin{align*}
F(z,s) := \sum_{ j \in \mathbb{Z} } \frac{ e^{ 2 \pi i j s }}{ z + 2 \pi i j } =  \frac{ e^{ - zs}}{ 1 - e^{ - z}}  \qquad \text{and} \qquad \sum_{ j \in \mathbb{Z} +1/2 } \frac{ e^{ 2 \pi i j s }}{ z + 2 \pi i j } =  \frac{ e^{ - zs}}{ 1 + e^{ - z}} 
\end{align*}
As for the $s = 0$ case, we have the identity
\begin{align*}
F(z,0) := \sum_{ j \in \mathbb{Z}} \frac{ 1}{ z + 2 \pi i j } = \frac{1}{2} \frac{ 1 + e^{-z}}{ 1 - e^{ -z}} .
\end{align*}
\end{lemma}

\begin{proof}
%%%%%%%%%%%%%%%%%%%%%%%%%
%%%%%%%%%%%%%%%%%%%%%%%%%
%%%%%%%%%%%%%%%%%%%%%%%%%          DUBIOUS PROOF
%%%%%%%%%%%%%%%%%%%%%%%%%
%%%%%%%%%%%%%%%%%%%%%%%%%
Let us introduce the Poisson summation formula, following Katznelson \cite[Chapter VI]{katznelson}. Let $f \in \mathcal{L}^1(\mathbb{R})$ be an integrable complex-valued function on the real line, and define the $\varphi:[0,2 \pi) \to \mathbb{C}$ by 
\begin{align*}
\varphi(t) := 2 \pi \sum_{ j \in \mathbb{Z}} f(t + 2 \pi j ).
\end{align*}
It is easily verified using the integrability of $f$ that $\varphi$ is integrable on $[0,2\pi)$. Furthermore, we associate with $\varphi$ and/or $f$ the Fourier coefficients
\begin{align*}
\hat{f}(n) := \int_{-\infty}^\infty f(x) e^{ -  i n x} \mathrm{d} x = \frac{1}{2 \pi } \int_0^{2 \pi} \varphi(t) e^{ - i n t } \mathrm{d} t. 
\end{align*}
Then provided $t = 0$ is a point of continuity for $\varphi(t)$, the \emph{Poisson summation formula} states that
\begin{align*}
\varphi(0) := 2 \pi\sum_{j \in \mathbb{Z}} f( 2 \pi j) = \sum_{ j \in \mathbb{Z}} \hat{f}(j).
\end{align*}
We now apply the Poisson summation formula to our problem. 
Consider now defining for $s \in [0,1], a >0, b \in \mathbb{R}$ the integrable function
\begin{align*}
f_{s,a,b}( x) := \frac{ e^{i sx  } }{ a^2 + (x + b)^2 }.
\end{align*}
Using the Fourier coefficients of the Cauchy distribution, a calculation tells us that the Fourier coefficients associated with $f_{s,a,b}$ are given by 
\begin{align*}
\hat{f}_{s,a,b}(j) = \int_{-\infty}^\infty \frac{ e^{2 \pi i (s-j)x } }{ a^2 + (x + b)^2 } \mathrm{d} x = \frac{ \pi e^{ - i (s-j)b} }{ a} e^{ - a |j-s| }.
\end{align*} 
In particular, provided $\varphi_{s,a,b}(t) := 2 \pi \sum_{ j \in \mathbb{Z}} f_{s,a,b}(t + 2 \pi j)$ is continuous at zero, using the Poisson summation formula to obtain the first inequality, and summing the geometric series to obtain the second, for $a > 0, b \in \mathbb{R}$ and $s \in (0,1)$ we have 
\begin{align} \label{eq:strauss}
\frac{1}{2\pi} \varphi_{s,a,b}(0) := \sum_{ j \in \mathbb{Z}} \frac{ e^{2 \pi j  i s   } }{ a^2 + (2 \pi j + b)^2 } = \frac{1}{2\pi } \sum_{ j \in \mathbb{Z}} \frac{ \pi e^{ - i (s-j)b} }{ a} e^{ - a |j-s| } = \frac{1}{2a} \left( \frac{ e^{ - s w}}{ 1 - e^{ - w}} + \frac{ e^{ - (1-s)\bar{w} } }{ 1 - e^{ - \bar{w}} } \right),
\end{align}
where $w := a + ib $. Noting that 
\begin{align*}
\varphi_{s,a,b}(t) = e^{ i st } \varphi_{s,a,t+b}(0),
\end{align*}
it is easily verified that for $s > 0$, the candidate formula for $\varphi_{s,a,b}(t)$ is continuous at $t = 0$, and hence the application of the Poisson summation abvoe was valid. 

\vspace{5mm}
We now turn to the sum in question. Setting $z = a + i b $, by reifying the denominator we have
\begin{align} \label{eq:prokofiev}
\sum_{ j \in \mathbb{Z} } \frac{ e^{ 2 \pi i j s }}{ z + 2 \pi i j }  &= \bar{z} \sum_{ j \in \mathbb{Z}}  \frac{ e^{ 2 \pi i j s }}{ a^2 + (2 \pi j + b)^2  }  -  \sum_{ j \in \mathbb{Z}}  \frac{ 2 \pi j i e^{ 2 \pi i j s}}{ a^2 + (2 \pi j + b)^2  }  \nonumber \\
&= \frac{1}{2\pi}  \bar{z} \varphi_{s,|a|,b}(0) -  \frac{1}{2\pi}  \frac{\partial}{ \partial s} \varphi_{s,|a|,b}(0),
\end{align}
where $\varphi_{s,|a|,b}(0)$ is as in \eqref{eq:strauss}, and the interchange in the order of differentiation and summation is justified by the fact that the series 
$\sum_{ j \in \mathbb{Z}} \frac{ j e^{ 2 \pi i j s }}{ x^2 + (2 \pi j + y)^2}$ converges; this last point may be proved in much the same method as that used in Lemma \ref{lem:tech}. 

Note that setting $z = a + ib $ and $w = |a| + i b$, by a brief calculation we have 
\begin{align} \label{eq:jarvi}
\left( \frac{ e^{ - s w}}{ 1 - e^{ - w}} + \frac{ e^{ - (1-s)\bar{w} } }{ 1 - e^{ - \bar{w}} } \right) =\left( \frac{ e^{ - s z}}{ 1 - e^{ - z}} + \frac{ e^{ - (1-s)\bar{z} } }{ 1 - e^{ - \bar{z}} } \right),
\end{align}
with the understanding that both sides are zero when $a = 0$. 

In particular, a calculation using \eqref{eq:prokofiev}, \eqref{eq:strauss} and \eqref{eq:jarvi} tells us that for $s \in (0,1)$ and $a,b \in \mathbb{R}$ we have

%%%%%%%%%%%%%%%%%%%%%%%%%
%%%%%%%%%%%%%%%%%%%%%%%%%
%%%%%%%%%%%%%%%%%%%%%%%%%          DUBIOUS PROOF
%%%%%%%%%%%%%%%%%%%%%%%%%
%%%%%%%%%%%%%%%%%%%%%%%%%

We can compute $F_{x + iy}(s)$ by the Poisson summation formula. Namely, for $f \in L^1(\mathbb{R})$ we have \\$\sum_{ j \in \mathbb{Z}} e^{ 2 \pi i j s } f( 2 \pi j ) = \frac{1}{2 \pi} \sum_{ k \in \mathbb{Z}} G( s+ k ) $ where $G(t) := \int_{ -\infty}^\infty f(y) e^{ i y t} \mathrm{d} y$. In particular, in our setting we have $f(x) = \frac{1}{ a^2 + (b + x)^2}$ and $G(t) := \frac{ \pi e^{ -  i b t}}{a} e^{ - | a w t | }$. It follows 
\begin{align*}
F_{a,b}(s) &= \frac{1}{2\pi} \sum_{ k \in \mathbb{Z}} \frac{ \pi e^{ - i b ( s + k ) } }{ a} e^{ - a | s + k | }
\end{align*}
which following a computation reduces to
\begin{align*}
F_{a,b}(s) &= \frac{ 1 }{ 2 \mathrm{Re}(z) } \left( \frac{ e^{ - s z}}{ 1 - e^{ - z}} + \frac{ e^{ - (1-s)\bar{z} } }{ 1 - e^{ - \bar{z}} } \right).
\end{align*}
A computation gives the result.

As for the $s =0$ case, we have 
\begin{align*}
\sum_{j \in \mathbb{Z}} \frac{1}{z+ 2 \pi i j } &= \frac{1}{z} + \sum_{ j \geq 1} \left( \frac{1}{z+ 2 \pi i j } + \frac{1}{z-2\pi i j } \right)\\
&= \frac{1}{z} + \sum_{ j \geq 1} \frac{ (z/2) }{ (z/2)^2 + \pi^2 j^2 }\\
&= \frac{1}{2} \mathrm{coth}( z/2) = \frac{1}{2} \frac{ 1 + e^{ - z}}{ 1 - e^{ - z}},
\end{align*}
where the penultimate equality above follows from the well known series representation 
for the cotangent function. (For information on this representation the reader is directed to the interesting discussion on the \emph{Herglotz trick} in \cite[Chapter 25]{AZ}.)

\end{proof}

Let us make a brief interprative remark on Lemma \ref{lem:fourier}. While the complex-valued function given by $F(z,s) := \sum_{ j \in \mathbb{Z} } \frac{ e^{ 2 \pi i j s}}{ z + 2 \pi i j } $ on the torus is not continuous at $s = 0$, we have \[F(z,0) = \frac{1}{2} \lim_{ \varepsilon \to 0} ( F(z, \varepsilon) + F(z, 1-\varepsilon) ).\]

We now turn to our next technical lemma, which, by setting $z = T(p+w)$, states when $s \in (0,1)$, the equality in \eqref{eq:puresum} does indeed hold. (We will see in a moment that this is \emph{not} the case when $s = 0$.)

\begin{lemma}  \label{lem:techs}
Let $s \in (0,1)$. Then with $F(z,s)$ as in Lemma \ref{lem:tech1} we have
 \begin{align} \label{eq:puresum2}
\lim_{m \to \infty} \frac{1}{m} \sum_{ j = -m/2+1}^{m/2}   \frac{ e^{  2 \pi i \frac{ j + \theta_1/2}{m} \lfloor sm \rfloor } }{1 -  \exp \left\{ - 2 \pi i  \frac{ j + \theta_1/2}{m} \right\} +  \frac{z}{m} } = \frac{ e^{ - z s } }{ 1 + (-1)^{\theta_1 - 1} e^{-z} }.
\end{align}

\end{lemma}

\begin{proof}
Let $s \in (0,1)$. Here, we may write  
\begin{align} \label{eq:repup}
 \frac{1}{m} \sum_{ j = -m/2+1}^{m/2} \frac{ e^{  2 \pi i \frac{ j + \theta_1/2}{m} \lfloor sm \rfloor } }{1 -  \exp \left\{ - 2 \pi i  \frac{ j + \theta_1/2}{m} \right\} +  \frac{z}{m} } = e^{ \pi i \theta_1 \lfloor sm \rfloor/m } \sum_{ j = - m/2 +1 }^{m/2} a_{j,m} b_m^j
\end{align}
where $a_{j,m} = a_{j,m}(z,\theta_1)$ and $b_m = b_m(s)$ are given by 
\begin{align*}
a_{j,m} = \frac{1}{ m\left( 1 - \exp( - 2 \pi i \frac{j+\theta_1/2}{m} ) + z/m \right) } \qquad \text{and} \qquad b_m = e^{ 2 \pi i \lfloor sm \rfloor/m}.
\end{align*}
For $j < m/2$, set $c_{j,m} := a_{j,m} - a_{j+1,m}$. It is easily verified that there is a constant $C = C(z)$ such that for all $m$ and all $-\frac{m}{2}+1\leq j \neq 0 \leq \frac{m}{2}$ we have
\begin{align} \label{eq:nicebounds}
|a_{j,m}| \leq C/j \qquad \text{and} \qquad |c_{j,m} | \leq C/j^2.
\end{align}
Now using the definition of $c_{j,m}$, we have 
\begin{align}
a_{j,m} := a_{m/2,m} + \sum_{ \ell = j}^{m/2-1} c_{j,m}.
\end{align}
It follows that for any integer $m_0 \geq 1$ we have 
\begin{align} \label{eq:christ}
\sum_{ j = m_0}^{m/2} a_{j,m} b^j &= \sum_{ j = m_0}^{m/2} b^j \left( a_{m/2,m} + \sum_{ \ell = j}^{m/2-1} c_{\ell,m} \right) \nonumber \\
&= a_{m/2,m} \sum_{ j = m_0}^{m/2} b^j + \sum_{ \ell = m_0}^{m/2} c_{\ell,m}  \sum_{ j = m_0}^{\ell} b^j .
\end{align}
It is easily verified from the definition of $b_m$ that summing the geometric series that for any integers $m_1, m_2$, $\left| \sum_{ j = m_1}^{m_2} b_m^j \right| \leq \frac{2}{|1-b_m| } \leq C(s)$ for some constant $C(s)$ depending on $s$ but not $m$. Using this fact with \eqref{eq:christ} to obtain the first inequality below, and then \eqref{eq:nicebounds} to obtain the second, we have 
\begin{align} \label{eq:christ2}
\left| \sum_{ j = m_0}^{m/2} a_{j,m} b^j \right| \leq C(s) \left( |a_{m/2},m| + \left| \sum_{ \ell = m_0}^{m/2} c_{\ell,m}   \right| \right) \leq C_1 /m_0 
\end{align}
for some constant $C_1 = C_1 (s,z)$ not depending on $m$ or $m_0$. A similar argument extending consideration to negative $j$ tells us that the contribution to the sum in \eqref{eq:repup} for $|j| \geq m_0$ is $O(1/m_0)$. That is, there is a constant $C_2 = C_2(s,z)$ such that for all $m,m_0$ we have
\begin{align} \label{eq:banane}
\left| e^{ \pi i \theta_1 \lfloor sm \rfloor/m } \sum_{ -m/2 +1 \leq j \leq m/2, |j| \geq m_0 } a_{j,m} b_m^j \right| \leq C_2 /m_0.
\end{align}
Now write $d_{j,m} := e^{ \pi i \theta_1 \lfloor sm \rfloor/m } a_{j,m} b_m^j$ for the summands. Then it is easily verified that there is a constant $C_3 = C_3(s,z)$ such that 
\begin{align*}
\left| d_{j,m} - \frac{ e^{ 2 \pi i (j + \theta_1/2) s } }{ 2 \pi i (j + \theta_1/2) + z ) } \right| \leq C_3 j^2/m.
\end{align*}
It follows that for some $C_4 = C_4(s,z)$ independent of $m_0$ and $m$ we have 
\begin{align} \label{eq:oranje}
\left| \sum_{ |j| \leq m_0 } d_{j,m} - \sum_{ |j| \leq m_0 } \frac{ e^{ 2 \pi i (j + \theta_1/2) s } }{ 2 \pi i (j + \theta_1/2) + z ) } \right| \leq C_4 m_0^3/m.
\end{align}
Combining \eqref{eq:banane} and \eqref{eq:oranje}, with $O$ terms that depend on $s$ and $z$ but are independent of $m$ and $m_0$ we have
\begin{align*}
\sum_{ j = -m/2+1}^{m/2} d_{j,m} = \sum_{ |j| \leq m_0 }  \frac{ e^{ 2 \pi i (j + \theta_1/2) s } }{ 2 \pi i (j + \theta_1/2) + z ) }  + O(m_0^3/m) + O(1/m_0).
\end{align*}
In particular, setting $m_0 = m^{1/4}$ we have 
\begin{align} \label{eq:1984}
\sum_{ j = -m/2+1}^{m/2} d_{j,m} = \sum_{ |j| \leq m^{1/4} }  \frac{ e^{ 2 \pi i (j + \theta_1/2) s } }{ 2 \pi i (j + \theta_1/2) + z ) }  + O(m^{-1/4}).
\end{align}
Combining \eqref{eq:1984} with \eqref{eq:boundc} we have 
\begin{align} \label{eq:1985}
\sum_{ j = -m/2+1}^{m/2} d_{j,m} = \sum_{ j \in \mathbb{Z} }  \frac{ e^{ 2 \pi i (j + \theta_1/2) s } }{ 2 \pi i (j + \theta_1/2) + z ) }  + O(m^{-1/4}).
\end{align}
In other words, thanks to \eqref{eq:1985}, with $F(z,s)$ as in Lemma \ref{lem:tech1} we have
 \begin{align*} 
\lim_{m \to \infty} \frac{1}{m} \sum_{ j = -m/2+1}^{m/2}   \frac{ e^{  2 \pi i \frac{ j + \theta_1/2}{m} \lfloor sm \rfloor } }{1 -  \exp \left\{ - 2 \pi i  \frac{ j + \theta_1/2}{m} \right\} +  \frac{z}{m} } = e^{ \pi i \theta_1 s } F(z + \pi i \theta_1 , s)
\end{align*}
The result follows from noting that, due to Lemma \ref{lem:fourier} we have 
\begin{align*}
e^{ \pi i \theta_1 s } F(z + \pi i \theta_1 , s) = e^{ \pi \theta_1 i s } \frac{ e^{ - (z + \pi i \theta_1 ) s } }{ 1 - e^{ - z - \pi i \theta_1 } } = \frac{ e^{ - z s } }{ 1 + (-1)^{\theta_1 } e^{ - z} }.
\end{align*} 

\end{proof} 

The next lemma is the analogue of the previous lemma for $s = 0$. This computation turns out to be more involved, since the lack of periodic term $e^{ 2 \pi i j s }$ means that all terms $- m/2 +1 \leq j \leq m/2$ contribute to the asymptotics, and hence extra terms emerge as contour integrals.

\begin{lemma} \label{lem:tech0}
Let the real part of $z$ be non-zero. Then for $\delta \in \{0,1\}$ we have 
 \begin{align} \label{eq:puresum2}
 &\lim_{m \to \infty} \frac{1}{m} \sum_{ j = -m/2+1}^{m/2}   \frac{e^{-2\pi i \delta \frac{j+\theta_1/2}{m}}  }{1 -  \exp \left\{ - 2 \pi i  \frac{ j + \theta_1/2}{m} \right\} +  \frac{z}{m} } = \frac{ \left( (-1)e^{-z} \right)^{\delta} }{1 - (-1)^{\theta_1}e^{-z}}.
\end{align}

\end{lemma}

\begin{proof}
Through this proof, all $O$ terms will depend implicitly on $z,\theta_1,\delta$, but be independent of $j/m$. 

Define the function
\begin{align*}
f_{\delta,m}(t) := \frac{ e^{ - 2 \pi i \delta t }}{m \left( 1 - e^{ - 2\pi i t } + z/m\right)}.
\end{align*}
Then we may write 
\begin{align} \label{eq:ova1}
I_m := \frac{1}{m} \sum_{ j = -m/2+1}^{m/2}   \frac{e^{-2\pi i \delta \frac{j+\theta_1/2}{m}}  }{1 -  \exp \left\{ - 2 \pi i  \frac{ j + \theta_1/2}{m} \right\} +  \frac{z}{m} }  = G_\delta(z/m) + B_{\delta,m},
\end{align}
where for $\alpha \in \mathbb{C}$, $G_\delta(\alpha) := \int_0^1 \frac{e^{ - 2 \pi i \delta \phi} }{1 + e^{ - 2 \pi i \phi } + \alpha } \mathrm{d} \phi$, and $B_{\delta,m} = \sum_{ j = -m/2+1}^{m/2} q_{\delta,j,m}$, with
\begin{align*}
q_{\delta,j,m} := \int_{-1/2}^{1/2} f_{\delta,m}(j+\theta_1/2) - f_{\delta,m}(j+\theta_1/2+u) \mathrm{d}u.
\end{align*}
Writing $G_\delta(\alpha)$ as a contour integral, and by considering the poles of the complex function\\
$f_\alpha(w) := w^{-(\delta+1)}( 1 -  w^{-1} + \alpha)^{-1}$ in the separate cases $\delta=0$ and $\delta=1$, we see that provided the modulus of $1+\alpha$ is not equal to one we have 
\begin{align*}
G_\delta(\alpha) = \oint_{S^1} \frac{ \mathrm{d}w}{ 2 \pi i w} \frac{w^{-\delta}}{1 - 1/w + \alpha}  = 
\begin{cases}
\frac{1}{1+\alpha} \mathbf{1}_{ |1+\alpha|>0 }   \qquad &\text{if $\delta = 0$},\\
- \mathbf{1}_{|1+\alpha| < 0 }  \qquad &\text{if $\delta=1$}.
\end{cases}
\end{align*}
Now if the real part of $z$ is non-zero, then for all sufficiently large $m$, $1+z/m$ has modulus different to one. In particular, we have
\begin{align} \label{eq:ova2}
\lim_{m \to \infty} G_\delta(z/m) = \mathbf{1}_{ \mathrm{Re}(z) > 0 , \delta = 0 } - \mathbf{1}_{ \mathrm{Re}(z) < 0 , \delta=1}
\end{align}
We now turn to evaluating $B_{\delta,m}$. On the one hand, a brief calculation tells us that
\begin{align} \label{eq:sib1}
q_{\delta,j,m} = O(1/j^2).
\end{align}
(Here and throughout, the $O$ term is uniform on $j$ and $m$, but may depend on $z,\delta,\theta_1$.) 

On the other hand, $f_{\delta,m}(t) = \frac{1}{2 \pi i t + z} + O_z(t/m)$, we have  
\begin{align} \label{eq:sib2}
q_{\delta,j,m} = \frac{1}{2 \pi i (j + \theta_1/2 ) + z} - \int_{-1/2}^{1/2} \frac{ \mathrm{du} }{ 2 \pi i (j + \theta_1/2+ u ) + z } + O(j/m).
\end{align}
In particular, for $1 \leq m_0 < m$, by \eqref{eq:sib1} we have 
\begin{align} \label{eq:fedor}
\sum_{ |j|>m_0} q_{\delta,j,m} = O(1/m_0),
\end{align}
and by \eqref{eq:sib2} we have 
\begin{align} \label{eq:alpine}
\sum_{ |j| \leq m_0} q_{\delta,j,m} = \sum_{ |j| \leq m_0 } \frac{1}{2 \pi i (j + \theta_1/2 ) + z} - \int_{-1/2}^{1/2} \sum_{ |j| \leq m_0 } \frac{ \mathrm{du} }{ 2 \pi i (j + \theta_1/2+ u ) + z } + O(m_0^2/m).
\end{align}
Now by virtue of Lemma \ref{lem:tech1}, with $F(z,0) := \sum_{ j \in \mathbb{Z}} \frac{1}{2 \pi i j + z }$, developing \eqref{eq:alpine} further we have 
\begin{align} \label{eq:alpine2}
\sum_{ |j| \leq m_0} q_{\delta,j,m} &= F(z+\pi i \theta_1,0) - \int_{-1/2}^{1/2} F(z+\pi i \theta_1 + 2 \pi i u) + O\left( \frac{1}{m_0 - |z+u| } \right) \mathrm{d}u +  O(m_0^2/m) \nonumber \\
&= F(z+\pi i \theta_1,0) - \int_{-1/2}^{1/2} F(z+\pi i \theta_1 +  2 \pi i u) \mathrm{d}u + O(1/m_0) +   O(m_0^2/m),
\end{align}
where we recall that we permit our $O$ terms to depend on $z$.

Combining \eqref{eq:fedor} with \eqref{eq:alpine2}, and using the definition of $B_{\delta,m}$ to obtain the first equality below, and then setting $m_0  = \lfloor m^{1/3} \rfloor$ to obtain the second, we see that
\begin{align} \label{eq:puresum3}
B_{\delta,m} &= F(z+\pi i \theta_1,0) - \int_{-1/2}^{1/2} F(z+\pi i \theta_1 + 2 \pi i  u) + O(1/m_0) +   O(m_0^2/m) \nonumber \\
&= F(z+\pi i \theta_1,0) - \int_{-1/2}^{1/2} F(z+\pi i \theta_1 + 2 \pi i  u) \mathrm{d}u + O(m^{-1/3}).
\end{align}
Combining \eqref{eq:ova2} and \eqref{eq:puresum3} in \eqref{eq:ova1}, we see that
\begin{align} \label{eq:ova3}
\lim_{m \to \infty} I_m =  \mathbf{1}_{ \mathrm{Re}(z) > 0 , \delta = 0 } - \mathbf{1}_{ \mathrm{Re}(z) < 0 , \delta=1} +  F(z+\pi i \theta_1,0) - \int_{-1/2}^{1/2} F(z+\pi i \theta_1 + 2 \pi i  u, 0 ) \mathrm{d}u. 
\end{align}
We now evaluate the integral in \eqref{eq:ova3}. We note that thanks to the periodicity $F(z+2 \pi i k, 0) = F(z,0)$ for all integers $k$ and all $z$ in $\mathbb{C}$, writing $z = x + iy$ we have $\int_{-1/2}^{1/2} F(z+\pi i \theta_1 + 2 \pi i u , 0 ) \mathrm{d}u = \int_{-1/2}^{1/2} F(x + 2 \pi i u , 0 ) \mathrm{d}u$. Using the formula for $F(z,0)$ in Lemma \ref{lem:fourier}, and then performing the contour integration, we have 
\begin{align} \label{eq:andres} 
 \int_{-1/2}^{1/2} F(x + 2 \pi i u , 0 ) \mathrm{d}u &= \int_{-1/2}^{1/2}  \frac{1}{2} \frac{ 1 + e^{ - x - 2 \pi i u }}{ 1 - e^{ - x - 2 \pi i u } } \mathrm{d}u \nonumber \\
&=  \oint_{|w|=1}\frac{ \mathrm{d}w}{2\pi i w} \frac{1}{2} \frac{ 1 + e^{ - x}/w}{ 1 - e^{ -x}/w } \mathrm{d}u \nonumber\\
&= \frac{1}{2} - \mathbf{1}_{x < 0 },
\end{align} 
where the term $\frac{1}{2}$ is due to the contribution from the pole at zero, and $- \mathbf{1}_{x < 0}$ is due to the pole at $e^x$ (which only lies inside the contour $\{ |w| = 1 \}$ when $x = \mathrm{Re}(z)$ is negative). Plugging \eqref{eq:andres} into \eqref{eq:ova3}, and using the definition of $F(z,0)$ in Lemma \ref{lem:fourier}, we obtain 
\begin{align} \label{eq:ova4}
\lim_{m \to \infty} I_m =  \mathbf{1}_{ \mathrm{Re}(z) > 0 , \delta = 0 } - \mathbf{1}_{ \mathrm{Re}(z) < 0 , \delta=1} + \frac{1}{2} \frac{1 +(-1)^{\theta_1} e^{-z}}
{1 - (-1)^{\theta_1} e^{-z}}  - \frac{1}{2} + \mathbf{1}_{\mathrm{Re}(z) < 0 },
\end{align}
which simplifies further to 
\begin{align} \label{eq:ova5}
\lim_{m \to \infty} I_m =  \mathbf{1}_{\delta = 0}  + \frac{1}{2} \frac{1 +(-1)^{\theta_1} e^{-z}}{1 - (-1)^{\theta_1} e^{-z}} - \frac{1}{2} .
\end{align}
A further calculation obtains the answer in \eqref{eq:puresum2}.

\end{proof}

We are now equipped to finally prove Theorem \ref{thm:inverselim}.
\begin{proof}[Proof of Theorem \ref{thm:inverselim}]
We begin with proving the equation \eqref{eq:inverselim}. Here, as remarked at the beginning of Section \ref{sec:inverselimproof}, without loss of generality we may assume $s \geq 0$. From here, we distinguish between the cases $s = 0$ and $s > 0$. The case $s > 0$ follows from \eqref{eq:lion} and Lemma \ref{lem:techs}. The case $s  = 0$ follows from \eqref{eq:lion} and the $\delta = 0$ case of Lemma \ref{lem:tech0}.

We now turn to \eqref{eq:inverselim0}. This equation follows by replacing $\lfloor sm \rfloor$ with $-1$ in \eqref{eq:lion}, and then using the $\delta=1$ case of Lemma \ref{lem:tech0}.
\end{proof}

We now tie our work together to prove Theorem \ref{thm:corr}.

\begin{proof}[Proof of Theorem \ref{thm:corr}]
Let $(x_i : i \in \mathcal{B} \sqcup \mathcal{O} \sqcup \mathcal{U} )$ be a collection of points on $\mathbb{T}_n$, and let $\tilde{x}_i$ and $\tilde{y}_i$ denote the collection of associated points on the doubly discrete torus $\mathbb{T}_{m,n}$ as defined at the beginning of Section \ref{sec:corrproof}. If $x_i = (t_i,h_i)$, then
\begin{align*}
\tilde{x}_i = ([[t_im]],h_i) \qquad \text{and} \qquad \tilde{y}_i = ([[t_im]]+\mathrm{1}_{i \in \mathcal{O}}, [h_i +\mathbf{1}_{i \in \mathcal{B}} ]).
\end{align*}
Then by \eqref{eq:inverse45} we have
\begin{align} \label{eq:scr}
\left(K^{1,1-\frac{\lambda}{m},\frac{T}{m}}_{m,n,\theta}\right)^{-1}(\tilde{y}_i,\tilde{x}_j) = \frac{1}{n} \sum_{ w^n = (-1)^{\theta_2}} w^{ - (h_j - [h_i +\mathbf{1}_{i \in \mathcal{B}}] ) } \frac{1}{m} \sum_{ j = 0}^{m-1}  \frac{ \exp \left\{ - 2 \pi i \frac{ j_1 + \theta_1 /2}{ m}\left( [[t_jm]] - [[t_jm]]- \mathrm{1}_{i \in \mathcal{O}} \right) \right\}  }{ 1 + \left( 1 - \frac{\lambda}{m} \right)  \exp \left\{ 2 \pi i  \frac{ j_1 + \theta_1 /2}{ m_1 } \right\} + \frac{Tw}{m}  }.
\end{align}
We now use Theorem \ref{thm:inverselim} to study the asymptotics of the term in \eqref{eq:scr} that depends on $m$. Multiplying the numerator and denominator by $\frac{1}{1 - \frac{\lambda}{m}}$ we see that we have 
\begin{align*}
 &\frac{1}{m} \sum_{ j = 0}^{m-1}  \frac{ \exp \left\{ - 2 \pi i \frac{ j_1 + \theta_1 /2}{ m}\left( [[t_jm]] - [[t_jm]]- \mathrm{1}_{i \in \mathcal{O}} \right) \right\}  }{ 1 + \left( 1 - \frac{\lambda}{m} \right)  \exp \left\{ 2 \pi i  \frac{ j_1 + \theta_1 /2}{ m_1 } \right\} + \frac{Tw}{m}  }\\
&=  \frac{1}{m} \sum_{ j = 0}^{m-1}  \frac{ \exp \left\{ - 2 \pi i \frac{ j_1 + \theta_1 /2}{ m}\left( [[t_jm]] - [[t_jm]]- \mathrm{1}_{i \in \mathcal{O}} \right) \right\}  }{ 1 +  \exp \left\{ 2 \pi i  \frac{ j_1 + \theta_1 /2}{ m_1 } \right\} + \frac{Tw + \lambda }{m}  } + O(1/m).
\end{align*}
We can now use Theorem \ref{thm:inverselim}. Note that since $ [[t_jm]] $ and $ [[t_im]]$ are both even, the quantity $[[t_jm]] - [[t_jm]]- \mathrm{1}_{i \in \mathcal{O}}$ has the same parity as $\mathrm{1}_{i \in \mathcal{O}}$. When $t_j \neq t_i$, by \eqref{eq:inverselim} we have
\begin{align} \label{eq:uni1}
\lim_{m \to \infty} \frac{1}{m} \sum_{ j = 0}^{m-1}  \frac{ \exp \left\{ - 2 \pi i \frac{ j_1 + \theta_1 /2}{ m}\left( [[t_jm]] - [[t_jm]]- \mathrm{1}_{i \in \mathcal{O}} \right) \right\}  }{ 1 + \left( 1 - \frac{\lambda}{m} \right)  \exp \left\{ 2 \pi i  \frac{ j_1 + \theta_1 /2}{ m_1 } \right\} + \frac{Tw}{m}  } =  (-1)^{ \mathrm{1}_{i \in \mathcal{O}} +  \theta_1 \mathbf{1}_{ t_j < t_i } } \frac{ e^{ - (Tw+\lambda)[t_j - t_i] } }{ 1 - (-1)^{\theta_1} e^{-(Tw+\lambda) } }  \qquad t_j \neq t_i. 
\end{align}
We study the case $t_j = t_i$, taking care to distinguish between the cases $i \in \mathcal{O}$ and $i \notin \mathcal{O}$. If $i \notin \mathcal{O}$, then we are simply in the setting of \eqref{eq:inverselim} again, this time with $s = 0$, and accordingly we have
\begin{align}\label{eq:uni2}
\lim_{m \to \infty} \frac{1}{m} \sum_{ j = 0}^{m-1}  \frac{ \exp \left\{ - 2 \pi i \frac{ j_1 + \theta_1 /2}{ m}\left( [[t_jm]] - [[t_jm]]- \mathrm{1}_{i \in \mathcal{O}} \right) \right\}  }{ 1 + \left( 1 - \frac{\lambda}{m} \right)  \exp \left\{ 2 \pi i  \frac{ j_1 + \theta_1 /2}{ m_1 } \right\} + \frac{Tw}{m}  } = \frac{ 1 }{ 1 - (-1)^{\theta_1} e^{-(Tw+\lambda) } } \qquad t_j = t_i, i \notin \mathcal{O}.
\end{align}
If $t_j = t_i$ and $i \in \mathcal{O}$, this time we are in the setting of \eqref{eq:inverselim0} and we have
\begin{align}\label{eq:uni3}
\lim_{m \to \infty} \frac{1}{m} \sum_{ j = 0}^{m-1}  \frac{ \exp \left\{ - 2 \pi i \frac{ j_1 + \theta_1 /2}{ m}\left( [[t_jm]] - [[t_jm]]- \mathrm{1}_{i \in \mathcal{O}} \right) \right\}  }{ 1 + \left( 1 - \frac{\lambda}{m} \right)  \exp \left\{ 2 \pi i  \frac{ j_1 + \theta_1 /2}{ m_1 } \right\} + \frac{Tw}{m}  } =  (-1)^{ \theta_1 + 1  } \frac{ e^{ - (Tw+\lambda) } }{ 1 - (-1)^{\theta_1} e^{-(Tw+\lambda) } }  \qquad t_j = t_i, i \in \mathcal{O}.
\end{align}
The equations \eqref{eq:uni1}, \eqref{eq:uni2} and \eqref{eq:uni3} may be unified as 
\begin{align} \label{eq:uni1}
\lim_{m \to \infty} \frac{1}{m} \sum_{ j = 0}^{m-1}  \frac{ \exp \left\{ - 2 \pi i \frac{ j_1 + \theta_1 /2}{ m}\left( [[t_jm]] - [[t_jm]]- \mathrm{1}_{i \in \mathcal{O}} \right) \right\}  }{ 1 + \left( 1 - \frac{\lambda}{m} \right)  \exp \left\{ 2 \pi i  \frac{ j_1 + \theta_1 /2}{ m_1 } \right\} + \frac{Tw}{m}  } =  (-1)^{ \mathrm{1}_{i \in \mathcal{O}} +  \theta_1 \mathbf{1}_{ t_j <_i t_i } } \frac{ e^{ - (Tw+\lambda)[t_j - t_i]_i } }{ 1 - (-1)^{\theta_1} e^{-(Tw+\lambda) } }  .
\end{align}
where $\{ t' <_i t \} = \{ t' < t \} \cup \{ t' = t, i \in \mathcal{O} \}$ and $[t'-t]_i = t'-t+\mathrm{1}_{t' < t} + \mathrm{1}_{i \in \mathcal{O}}\mathrm{1}_{t'=t} $. 

Plugging \eqref{eq:uni1} into \eqref{eq:scr} we have
\begin{align} \label{eq:scr2}
\lim_{m \to \infty} 
\left(K^{1,1-\frac{\lambda}{m},\frac{T}{m}}_{m,n,\theta}\right)^{-1}(\tilde{y}_i,\tilde{x}_j) =(-1)^{ \mathrm{1}_{i \in \mathcal{O}} +  \theta_1 \mathbf{1}_{ t_j <_i t_i } }   \frac{1}{n} \sum_{ w^n = (-1)^{\theta_2}} w^{ - (h_j - [h_i +\mathbf{1}_{i \in \mathcal{B}}] ) } \frac{ e^{ - (Tw+\lambda)[t_j - t_i]_i } }{ 1 - (-1)^{\theta_1} e^{-(Tw+\lambda) } }.
\end{align}
Note that $[h_i + \mathrm{1}_{i \in \mathcal{B}}] = h_i + \mathrm{1}_{i \in \mathcal{B}} - n \mathrm{1}_{i \in \mathcal{B}, h_i = n-1}$. In particular, for $w^n = (-1)^{\theta_2}$ we have
\begin{align*}
w^{ - (h_j - [h_i +\mathbf{1}_{i \in \mathcal{B}}] ) }  = (-1)^{\theta_2  \mathrm{1}_{i \in \mathcal{B}, h_i = n-1}} w^{ \mathrm{1}_{i \in \mathcal{B}}  +h_i-h_j}.
\end{align*}
In summary,
\begin{align} \label{eq:scr3}
\lim_{m \to \infty} 
\left(K^{1,1-\frac{\lambda}{m},\frac{T}{m}}_{m,n,\theta}\right)^{-1}(\tilde{y}_i,\tilde{x}_j) =(-1)^{ \mathrm{1}_{i \in \mathcal{O}} +  \theta_1 \mathbf{1}_{ t_j <_i t_i } + \theta_2  \mathrm{1}_{i \in \mathcal{B}, h_i = n-1} }   \frac{1}{n} \sum_{ w^n = (-1)^{\theta_2}} \frac{ w^{ \mathrm{1}_{i \in \mathcal{B}}  +h_i-h_j} e^{ - (Tw+\lambda)[t_j - t_i]_i } }{ 1 - (-1)^{\theta_1} e^{-(Tw+\lambda) } }.
\end{align}
Let $\{ x_i : i \in \mathcal{B} \sqcup \mathcal{O} \sqcup \mathcal{U} \}$ now be a collection of points in $\mathbb{T}_n$, where $x_i = (t_i,h_i)$. Recall we have defined
\begin{align*}
\tilde{x}_i := ([[t_im]],h) \qquad \tilde{y}_i := ([[[t_im]] + \mathrm{1}_{i \in \mathcal{O}}],[h_i + \mathrm{1}_{i \in\mathcal{B}}]).
\end{align*}
Then as remarked at the beginning of Section \ref{sec:corrproof} we have
\begin{align} \label{eq:scran}
\mathbf{P}_n^{\alpha,\theta,T} ( \Gamma(x_i : i \in \mathcal{B} \sqcup \mathcal{O} \sqcup \mathcal{U} )) = \lim_{m \to \infty} m^{ \# \mathcal{B} }  P^{1,1-\frac{\lambda}{m},\frac{T}{m}}_{m,n} \left( \sigma(\tilde{x}^1) = \tilde{y}^1,\ldots,\sigma(\tilde{x}^p) = \tilde{y}^p \right) ,
\end{align}
Now according to \eqref{eq:beet44}
\begin{align} \label{eq:beet55}
P^{1,1-\frac{\lambda}{m},\frac{T}{m}}_{m,n} \left( \sigma(x^1) = y^1,\ldots,\sigma(x^p) = y^p \right) = & (-1)^{ \theta_1 g_1 + \theta_2 g_2}  \left( 1-  \frac{T}{m} \right)^{ \# \{ i : y^i - x^i = \mathbf{e}_1 \} }   \left( \frac{T}{m} \right)^{ \# \{ i : y^i - x^i = \mathbf{e}_2 \} } \nonumber \\
& \times   \det \limits_{ i,j = 1}^k \left(  \left(K^{\alpha,\beta,\gamma}_{m_1,m_2,\theta}\right)^{-1} ( y^i,x^j) \right).
\end{align}
Now $g_2 = \# \{ i \in \mathcal{B} : h_i = n-1 \}$, and for sufficiently large $m$, $g_1 = 0$ since $g_1$ is the number of $i \in \mathcal{O}$ such that $[[t_i m]] = m-1$, which clearly cannot be the case whenever $t_i < 1$ for large $m$.

Plugging \eqref{eq:beet55} into \eqref{eq:scran} and using these observations about $g_1$ and $g_2$ to obtain the first equality below, and then using \eqref{eq:scr3} to obtain the second we have
\begin{align} \label{eq:scran2}
&\mathbf{P}_n^{\alpha,\theta,T} ( \Gamma(x_i : i \in \mathcal{B} \sqcup \mathcal{O} \sqcup \mathcal{U} )) \nonumber \\
&=(-1)^{ \theta_2 \# \{ i \in \mathcal{B} : h_i = n-1 \} } T^{ \# \mathcal{B}}   \lim_{m \to \infty}   \det \limits_{ i,j = 1}^k \left(  \left(K^{1, 1-\frac{\lambda}{m},\frac{T}{m} }_{m,n,\theta}\right)^{-1} ( \tilde{y}^i,\tilde{x}^j) \right)
 \nonumber \\
&=(-1)^{ \theta_2 \# \{ i \in \mathcal{B} : h_i = n-1 \} } T^{ \# \mathcal{B}} \det \limits_{ i, j \in \mathcal{B} \sqcup \mathcal{O} \sqcup \mathcal{U} }  \left( (-1)^{ \mathrm{1}_{i \in \mathcal{O}} +  \theta_1 \mathbf{1}_{ t_j <_i t_i } + \theta_2  \mathrm{1}_{i \in \mathcal{B}, h_i = n-1} }   \frac{1}{n} \sum_{ w^n = (-1)^{\theta_2}} \frac{ w^{ \mathrm{1}_{i \in \mathcal{B}}  +h_i-h_j} e^{ - (Tw+\lambda)[t_j - t_i]_i } }{ 1 - (-1)^{\theta_1} e^{-(Tw+\lambda) } }  \right) .
\end{align}
Rearranging to obtain the first equality below, and then noting that multiplying the $(i,j)^{\text{th}}$ entry by $e^{ \theta_i \pi i (t_j - t_i) } $ leaves the determinant unchanged to obtain the second, and then using the definition $[t_j - t_i ] := t_j - t_i + \mathrm{1}_{t_j < t_i} + \mathrm{1}_{i \in \mathcal{O} , t_j = t_i } = t_j - t_i + \mathrm{1}_{t_j <_i t_i} $ to obtain the third we have 
\begin{align}
&\mathbf{P}_n^{\alpha,\theta,T} ( \Gamma(x_i : i \in \mathcal{B} \sqcup \mathcal{O} \sqcup \mathcal{U} )) \nonumber \\
&=  \det \limits_{ i, j \in \mathcal{B} \sqcup \mathcal{O} \sqcup \mathcal{U} }  \left( (-1)^{ \mathrm{1}_{i \in \mathcal{O}} +  \theta_1 \mathbf{1}_{ t_j <_i t_i }}   \frac{T^{\mathrm{1}_{i \in \mathcal{B}}} }{n} \sum_{ w^n = (-1)^{\theta_2}} \frac{ w^{ \mathrm{1}_{i \in \mathcal{B}}  +h_i-h_j} e^{ - (Tw+\lambda)[t_j - t_i]_i } }{ 1 - (-1)^{\theta_1} e^{-(Tw+\lambda) } }  \right) \nonumber \\ 
&=  \det \limits_{ i, j \in \mathcal{B} \sqcup \mathcal{O} \sqcup \mathcal{U} }  \left( e^{ \theta_i \pi i (t_j - t_i) }  (-1)^{ \mathrm{1}_{i \in \mathcal{O}} +  \theta_1 \mathbf{1}_{ t_j <_i t_i }}   \frac{T^{\mathrm{1}_{i \in \mathcal{B}}} }{n} \sum_{ w^n = (-1)^{\theta_2}} \frac{ w^{ \mathrm{1}_{i \in \mathcal{B}}  +h_i-h_j} e^{ - (Tw+\lambda)[t_j - t_i]_i } }{ 1 - (-1)^{\theta_1} e^{-(Tw+\lambda) } }  \right)  \nonumber \\
&=  \det \limits_{ i, j \in \mathcal{B} \sqcup \mathcal{O} \sqcup \mathcal{U} }  \left( (-1)^{ \mathrm{1}_{i \in \mathcal{O}} } \frac{T^{\mathrm{1}_{i \in \mathcal{B}}} }{n} \sum_{ w^n = (-1)^{\theta_2}} \frac{ w^{ \mathrm{1}_{i \in \mathcal{B}}  +h_i-h_j} e^{ - (Tw+\theta_1 \pi i+ \lambda)[t_j - t_i]_i } }{ 1 - (-1)^{\theta_1} e^{-(Tw+\lambda) } }  \right) .
\end{align}
We now reconcile \eqref{eq:scran2} with Theorem \ref{thm:corr}. We need to check that the quantity inside the determinant agrees with either $K_{i,j}, H_{i,j}$ or $\delta_{i,j} - K_{i,j}$, depending on whether $i \in \mathcal{B}$, $\mathcal{O}$ or $\mathcal{U}$ respectively.

Now with $K_n^{\lambda,\theta,T}$ and $H_n^{\lambda,\theta,T}$ as in \eqref{eq:anotherK} and \eqref{eq:anotherH}, it is easily established that for $i \in \mathcal{O}$ we have 
\begin{align*}
(-1)^{ \mathrm{1}_{i \in \mathcal{O}} } \frac{T^{\mathrm{1}_{i \in \mathcal{B}}} }{n} \sum_{ w^n = (-1)^{\theta_2}} \frac{ w^{ \mathrm{1}_{i \in \mathcal{B}}  +h_i-h_j} e^{ - (Tw+\theta_1 \pi i+ \lambda)[t_j - t_i]_i } }{ 1 - (-1)^{\theta_1} e^{-(Tw+\lambda) } } = K_n^{\lambda,\theta,T}(x_i,x_j),
\end{align*}
and for $i \in \mathcal{B}$ we have
 \begin{align*}
(-1)^{ \mathrm{1}_{i \in \mathcal{O}} } \frac{T^{\mathrm{1}_{i \in \mathcal{B}}} }{n} \sum_{ w^n = (-1)^{\theta_2}} \frac{ w^{ \mathrm{1}_{i \in \mathcal{B}}  +h_i-h_j} e^{ - (Tw+\theta_1 \pi i+ \lambda)[t_j - t_i]_i } }{ 1 - (-1)^{\theta_1} e^{-(Tw+\lambda) } } = H_n^{\lambda,\theta,T}(x_i,x_j).
\end{align*}
It remains to establish that for $i \in \mathcal{U}$ we have 
 \begin{align*}
(-1)^{ \mathrm{1}_{i \in \mathcal{O}} } \frac{T^{\mathrm{1}_{i \in \mathcal{B}}} }{n} \sum_{ w^n = (-1)^{\theta_2}} \frac{ w^{ \mathrm{1}_{i \in \mathcal{B}}  +h_i-h_j} e^{ - (Tw+\theta_1 \pi i+ \lambda)[t_j - t_i]_i } }{ 1 - (-1)^{\theta_1} e^{-(Tw+\lambda) } } = \mathrm{1}_{x_i = x_j } - K_n^{\lambda,\theta,T}(x_i,x_j),
\end{align*}
or more explicitly, for $(t,h)$ and $(t',h')$ in $\mathbb{T}_n$, 
 \begin{align} \label{eq:turner}
 \frac{ 1 }{n} \sum_{ w^n = (-1)^{\theta_2}} \frac{ w^{ h-h'} e^{ - (Tw+\theta_1 \pi i+ \lambda)[t' - t] } }{ 1 - (-1)^{\theta_1} e^{-(Tw+\lambda) } } = \mathrm{1}_{(t,h) = (t',h')} +  \frac{1}{n} \sum_{w^n = (-1)^{\theta_2}} \frac{ w^{h-h'} e^{ - (\lambda+\theta_1 \pi i+Tw)([t'-t]+\mathrm{1}_{t'=t}) } }{ 1 - e^{-  (\lambda+\theta_1 \pi i +Tw) } } .
\end{align}
It is clear that \eqref{eq:turner} holds when $t'\neq t$. In the case $t ' = t$ \eqref{eq:turner} amounts to the equation:
 \begin{align} \label{eq:turner2}
 \frac{ 1 }{n} \sum_{ w^n = (-1)^{\theta_2}} \frac{ w^{ h-h'}  }{ 1 - e^{-  (\lambda+\theta_1 \pi i +Tw) } } = \mathrm{1}_{h'=h} +  \frac{1}{n} \sum_{w^n = (-1)^{\theta_2}} \frac{ w^{h-h'} e^{ - (\lambda+\theta_1 \pi i+Tw)} }{  1 - e^{-  (\lambda+\theta_1 \pi i +Tw) } } .
\end{align}
Writing
\begin{align*}
\frac{ e^{ - (\lambda+\theta_1 \pi i+Tw)} }{  1 - e^{-  (\lambda+\theta_1 \pi i +Tw) } } = - 1 + \frac{1}{  1 - e^{-  (\lambda+\theta_1 \pi i +Tw) } }, 
\end{align*}
we see that 
 \begin{align} \label{eq:turner3}
\mathrm{1}_{h'=h} +  \frac{1}{n} \sum_{w^n = (-1)^{\theta_2}} \frac{ w^{h-h'} e^{ - (\lambda+\theta_1 \pi i+Tw)} }{  1 - e^{-  (\lambda+\theta_1 \pi i +Tw) } } = \mathrm{1}_{h'=h} -  \frac{1}{n} \sum_{w^n = (-1)^{\theta_2}} w^{ h-h'} +  \frac{1}{n} \sum_{w^n = (-1)^{\theta_2}} \frac{ w^{h-h'} }{  1 - e^{-  (\lambda+\theta_1 \pi i +Tw) } } . 
\end{align}
Using the fact that for $h,h' \in \mathbb{Z}_n$,  $\frac{1}{n} \sum_{w^n = (-1)^{\theta_2}} w^{ h-h'}  = \mathrm{1}_{h'=h}$, \eqref{eq:turner2} now follows from \eqref{eq:turner3}.

That completes the proof of Theorem \ref{thm:corr}.
\end{proof}

\section*{Acknowledgements}
The author would like to thank Elia Bisi for his valuable comments.\\

\noindent
This research is supported by the EPSRC funded Project EP/S036202/1 \emph{Random fragmentation-coalescence processes out of equilibrium}.

\end{document}